\documentclass{elsarticle}

\usepackage{amsmath,amssymb,amscd,bbm,amsthm}
\usepackage[mathscr]{eucal}
\usepackage[all]{xy}
\usepackage[utf8]{inputenc}
\usepackage{hyperref}

 \newtheorem{thm}{Theorem}[section]
 \newtheorem{cor}[thm]{Corollary}
 \newtheorem{lem}[thm]{Lemma}
  
 \newtheorem{prop}[thm]{Proposition}
 \theoremstyle{definition}

 \newtheorem{rem}[thm]{Remark}
 \newtheorem{rems}[thm]{Remarks}
 \newtheorem{exa}[thm]{Example}
 
 \numberwithin{equation}{section}

\usepackage{depsymb}

\newcommand{\spec}{\usigma}
\newcommand{\Aspec}{\usigma_{\mathrm{ap}}}
\newcommand{\Pspec}{\usigma_{\mathrm{p}}}
\newcommand{\resol}{\urho}

\def\BMb{\calM_\mathrm{b}}
\def\Ee{\mathrm{E}}
\def\prO{\uOmega}

\newcommand{\vanish}[1]{\relax}

\newcommand{\beq}{\begin{equation}}
\newcommand{\eeq}{\end{equation}}
\newcommand{\defiff}{\stackrel{\text{\rm def}}{\Longleftrightarrow}}
\newcommand{\Ball}{\mathrm{B}}

\newcommand{\hs}{\hskip-0.1em}
\newcommand{\set}[1]{\hs\left[\,#1\,\right]}

\newcommand{\Clo}{\mathcal{C}}
\newcommand{\prfnoi}{\smallskip\noindent}

\newcommand{\emdf}{\bf}

\newcommand{\Id}{\mathrm{I}}

\newcommand{\suchthat}{\,\,|\,\,}

\newcommand{\N}{\mathbb{N}}

\newcommand{\R}{\mathbb{R}}
\newcommand{\C}{\mathbb{C}}

\newcommand{\K}{\mathbb{K}}

\newcommand{\eM}{\mathrm{M}}

\newcommand{\res}[1]{|_{#1}}

\newcommand{\Sum}[2][\relax]{%
 \ifx#1\relax \sideset{}{_{#2}}\sum 
 \else \sideset{}{^{#1}_{#2}}\sum
 \fi}

\newcommand{\car}{\mathbf{1}}

\newcommand{\konj}[1]{\overline{#1}}

\newcommand{\abs}[1]{\vert #1 \vert}

\DeclareMathOperator{\essran}{essran}

\newcommand{\Ce}{\mathrm{C}}
\newcommand{\Cc}{\mathrm{C}_\mathrm{c}}

\newcommand{\Ell}[1]{\mathrm{L}_{#1}}

\newcommand{\Meas}{\mathcal{M}}

\DeclareMathOperator{\bdd}{bdd}

\newcommand{\ohne}{\setminus}

\newcommand{\leer}{\emptyset}
\newcommand{\Pot}[1]{\mathfrak{P}\left(#1\right)}

\newcommand{\dann}{\Rightarrow}
\newcommand{\Dann}{\Longrightarrow}
\newcommand{\gdw}{\Leftrightarrow}

\newcommand{\nach}{\circ}

\DeclareMathOperator{\dom}{dom}
\DeclareMathOperator{\ran}{ran}

\newcommand{\cls}[1]{\overline{#1}}

\DeclareMathOperator{\supp}{supp}

\newcommand{\spann}{\mathrm{span}}

\newcommand{\tensor}{\otimes}

\newcommand{\BL}{\mathcal{L}}

\newcommand{\norm}[2][\relax]{%
   \ifx#1\relax \ensuremath{\lVert#2\rVert}
   \else \ensuremath{\left\Vert#2\right\Vert_{#1}}
   \fi}

\makeatletter
\newcommand{\sprod}[2]{\ensuremath{%
  \setbox0=\hbox{\ensuremath{#2}}
  \dimen@\ht0
  \advance\dimen@ by \dp0
  \left[ #1\rule[-\dp0]{0pt}{\dimen@}, #2\hspace{1pt}\right]}}
\newcommand{\bsprod}[2]{\ensuremath{%
  \setbox0=\hbox{\ensuremath{#2}}
  \dimen@\ht0
  \advance\dimen@ by \dp0
  \bigl[ #1\rule[-\dp0]{0pt}{\dimen@}, #2\hspace{1pt}\bigr]}}

 \makeatother
\newcommand{\dprod}[2]{\ensuremath{\langle#1,#2\rangle}}

\newcommand{\Borel}{\mathrm{Bo}}  
\newcommand{\Baire}{\mathrm{Ba}}

\usepackage{hyperref}
 
\begin{document}

\title{The Functional Calculus Approach to the Spectral Theorem}

\author[Markus Haase]{Markus Haase}

\address{%
Kiel University\\
Mathematisches Seminar\\
Ludewig-Meyn-Str.4\\
24118 Kiel, Germany}

\ead{haase@math.uni-kiel.de}

%\thanks{This work was completed while the author spend a research
%  sabbatical at UNSW in Sydney. The author is grateful to
%Fedor Sukochev (UNSW) for his kind invitation. Moreover, 
%the author acknowledges the financial support from the DFG under
%grant number  \dots}

%----------classification, keywords, date


\date{\today}
%----------additions
%\dedicatory{Dedicated to Wolfgang Arendt and Lutz Weis on the occasion
%of their retirement}
%%% ----------------------------------------------------------------------

\begin{abstract}
A consistent functional calculus approach to the spectral theorem
for strongly commuting normal operators on Hilbert spaces is
presented. In contrast to the common approaches using
projection-valued measures or multiplication operators, 
here the functional calculus is not treated as a subordinate but as the 
central concept. 

Based on five simple axioms for a ``measurable functional
calculus'', the theory of such calculi is developed in detail, 
including spectral theory, uniqueness results and construction
principles. Finally, the functional calculus form
of the spectral theorem is stated and proved, with some
proof variants being discussed.  
\end{abstract}

%%% ----------------------------------------------------------------------
\maketitle
%%% ----------------------------------------------------------------------
%\tableofcontents

\section{Introduction}\label{s.intro}

The spectral theorem (for normal or self-adjoint operators on a
Hilbert space) is certainly one of the most important results of
20th century mathematics. It comes in different forms, two of which
are the most widely used: the multiplication operator (MO) form 
and the one using projection-valued measures (PVMs). 
Associated with this variety is a discussion about
``What does the spectral theorem say?''(Halmos \cite{Halmos1963}), where the pro's and con's
of the different approaches are compared.

In this article, we would like to add a slightly 
different stance to this debate by  advocating
a consistent {\em functional calculus approach} 
to the spectral theorem. Since in any
exposition of the spectral theorem one also will
find results about functional calculus, some
words of explanation are in order.

\medskip
Let us start with the observation that whereas
multiplication operators and projection-valued
measures are well-defined mathematical objects,
the concept of a functional calculus
as used in the literature on the
spectral theorem is usually 
defined only implicitly.  One speaks
of {\em the}  functional calculus of a normal
operator (that is, the mapping whose
properties are listed in some particular theorem) 
rather than
of a functional calculus as an abstract concept.
As a result, such a concept 
remains a heuristic one,
and the concrete calculus 
associated with the  spectral theorem 
acquires and retains a subordinate status, being merely a derivate of the ``main'' formulations by
multiplication operators or projection--valued measures. 
(At this point, we should emphasize that 
we have the {\em full} functional calculus in mind,
comprising all measurable functions and not
just bounded ones.)
In practice, this expositional
dependence implies  that when using the functional
calculus (and one wants to use it all the time)
one always has to resort  
to one of its constructions.

In this respect,
the multiplication
operator version appears to have a slight advantage, since
deriving functional calculus properties
from facts about multiplication operators is
comparatively simple. (This is probably the reason
why eminent voices such as  Halmos \cite{Halmos1963} and
Reed--Simon \cite[VII]{ReedSimon1}  prefer multiplication
operators.)
However, this advantage is only virtual, since
the MO-version has two major drawbacks. Firstly,  a MO-representation
is not canonical and hence leads to the problem  whether functional calculus
constructions (square root, semigroup, logarithm etc.) are
independent of the chosen MO-representation. 
Secondly (and somehow related to
the first), the MO-version is hardly useful  for anything other  than
for deriving functional calculus properties. (For example,
it cannot be used in constructions, like that of
a joint (product) functional calculus.)

In contrast, an associated PVM is canonical and PVMs are
very good for  constructions,   but the description
of the functional calculus, in particular for unbounded functions, 
is cumbersome. And since one needs the functional calculus eventually,
every construction based on PVMs has, in order to be
useful, to be backed up by results about the functional
calculus associated with the new PVM.

\medskip
With the  present article we propose  
a ``third way'' of treating the spectral theorem, avoiding
the drawbacks of either one of the other approaches. 
Instead of treating the functional calculus as a
logically subordinate concept, we put it center stage
and make it our main protagonist.  Based on  
an axiomatic definition 
of a ``measurable functional calculus'', we
shall present a thorough development of the associated
theory entailing, in particular:
\begin{itemize}
\item general properties, 
constructions such as a pull-back and a push-forward
calculus  (Section \ref{s.mfc});

\item projection-valued measures, 
the role of null sets, 
the concepts of concentration and support
(Section \ref{s.pvm});

\item spectral theory (Section \ref{s.spc});

\item uniqueness (and commutativity) properties 
(Section \ref{s.uni});

\item construction principles (Section \ref{s.con}).
\end{itemize}
Finally, in Section \ref{s.spt}, we
state and prove ``our'' version of the spectral theorem,
which takes the following simple form (see Theorem \ref{spt.t.unb}).

\medskip
\noindent
{\bf Spectral Theorem:}\  
{\em Let $A_1, \dots, A_d$ be pairwise strongly commuting normal operators
on a Hilbert space $H$. Then there is
a unique Borel calculus $(\Phi, H)$ on
$\C^d$ such that $\Phi(\bfz_j) = A_j$
for all $j=1, \dots, d$.}  

\medskip
Here, we use a notion of strong commutativity which is formally 
different from that used by Schmüdgen in \cite{SchmuedgenUOH},
but is more suitable for our approach. 
In a final section we then show that both notions are 
equivalent.

\medskip
In order to advertise our approach,
let us point out some of its ``special features''. 
Firstly, the axioms for a measurable calculus are few and simple, 
and hence easy to verify. Restricted to bounded functions
they are just what one expects, but the main point is
that these axioms work for all measurable functions.

Secondly, the aforementioned axioms are
complete in the sense that each functional calculus property
which can be derived with the help of a MO-representation can
also be derived directly, and practically with the same effort,
 from the axioms.\footnote{Actually, with {\em less} effort,
since one saves the work for establishing the MO-representation,
Theorem \ref{spc.t.mult-rep}.}
This is of course not
a rigorous (meta)mathematical theorem, but a heuristic statement
stipulated by 
the exhaustive exposition we give. In particular, we demonstrate
that  many properties of multiplication operators (for example its
spectral properties) are 
consequences of the general theory, simply because
the multiplication operator calculus
satisfies the axioms\ of a measurable calculus (Theorem \ref{mfc.t.mul} and 
Corollary \ref{spc.c.mult}). 
%Of course, this is not surprising
%since any measurable functional calculus has a MO-representation
%(Theorem \ref{spc.t.mult-rep}). 

Thirdly, %and this should not be underestimated, 
the abstract functional calculus approach leads 
to a simple method for extending a calculus
from bounded to unbounded measurable functions (Theorem
\ref{con.t.ext-bdd}). This method, known as ``algebraic extension'' or
``extension by (multiplicative) regularization'',
is well-established  in general functional calculus theory for
unbounded operators such as sectorial operators or semigroup generators.
(See \cite{HaaseFC,Haase2020pre} and the references therein.) 
It has the enormous advantage that it is elegant and perspicuous,
and that it avoids cumbersome arguments with domains of operators,
which are omnipresent in the PVM-approach (cf.{ }Rudin's exposition
in \cite{RudinFA}).

\subsection*{Notation and Terminology}

We shall work generically over the scalar field $\K \in \{\R, \C\}$.
The letters $H,K$ usually denote Hilbert spaces,
the space of bounded linear operators from $H$ to $K$ is denoted by
$\BL(H;K)$, and $\BL(H)$ if $H= K$.

A (closed) linear subspace of $H\oplus K$ is called a (closed) 
{\em linear relation}. Linear relations
are called {\em multi-valued operators} in \cite[Appendix A]{HaaseFC}, and we use
freely the definitions
and results from that reference. In particular, we say
that a bounded operator $T \in \BL(H)$ {\em commutes} with
a linear relation $A$ if $TA \subseteq AT$, which is equivalent
to 
\[  (x,y) \in A \dann (Tx, Ty) \in A.
\]
%When $A \subseteq H \oplus K$ is a linear relation, we write
%\[   Ax = y \quad \text{as an abbreviation for}\quad (x,y) \in A.%
%\]
A linear relation is called an {\em operator} if it is functional,
i.e., it satisfies
\[  (x,y),\,(x,z)\in A \quad \dann\quad y=z.
\]
The {\em set of all closed linear operators} is 
\[ \Clo(H;K) := \{ A\subseteq H \oplus K \suchthat \text{$A$ is
closed and an operator}\},
\]
with $\Clo(H):= \Clo(H;H)$.

For the spectral theory of linear relations, we refer to
\cite[Appendix A]{HaaseFC}.
For a closed linear relation $A$ in $H$ we  denote by $\spec(A), 
\Pspec(A), \Aspec(A), \resol(A)$ the {\em spectrum, point spectrum,
approximate point spectrum} and {\em resolvent set}, respectively. The
{\em resolvent} of $A$ at $\lambda\in \resol(A)$ is
\[ R(\lambda, A) := (\lambda\Id - A)^{-1}.
\]

A {\em measurable space} is a pair $(X,\Sigma)$, where
$X$ is a set and  $\Sigma$ is a  $\sigma$-algebra of subsets of $X$.
A function $f: X \to \K$ is said to be {\em measurable} if it is
$\Sigma$-to-Borel measurable in the sense of measure theory. We abbreviate
\begin{align*} \Meas(X,\Sigma) & := \{ f: X \to \K \suchthat \text{$f$ measurable}\}
\quad\text{and}\\
\BMb(X,\Sigma)   & := \{ f\in \Meas(X,\Sigma) \suchthat \text{$f$
  bounded}\}.
\end{align*}
These sets are unital algebras with respect to the  pointwise
operations and unit element $\car$ (the function which is constantly
equal to $1$). 

Note that  $\BMb(X,\Sigma)$ is closed under {\em bp-convergence}, by which we
mean that if a sequence $(f_n)_n$ in $\BMb(X,\Sigma)$ converges
{\em b}oundedly (i.e., with $\sup_n \norm{f_n}_\infty <
\infty$) and {\em p}ointwise to a function $f$, then $f\in \BMb(X,\Sigma)$ as well. 

If the $\sigma$-algebra $\Sigma$
is understood, we simply write $\Meas(X)$ and $\BMb(X)$. 
If $X$ is a separable metric space, then by default we take $\Sigma=\Borel(X)$,
the {\em Borel} $\sigma$-algebra on $X$ generated by the family of
open subsets (equivalently: closed subsets, open/closed balls).

Let $(\Omega, \calF, \mu)$ be a measure space. A {\em null set} 
is any subset $A\subseteq \Omega$ such that there is $N \in \calF$
with $A \subseteq N$ and $\mu(N) = 0$.  A mapping $a : \dom(a) \to
X$, where $\dom(a) \subseteq \Omega$ and $(X, \Sigma)$ is any measurable set, 
is called {\em almost everywhere defined} if $\Omega \ohne \dom(a)$ 
is a null set. And it is 
called {\em essentially measurable}, if it is almost everywhere defined
and there is a measurable function $b: \Omega\to X$ such that
$\{x\in \dom(a) \suchthat a(x) \neq b(x)\}$ is a null set.

\section{Measurable Functional Calculus --- Definition and Basic Properties}\label{s.mfc}

\medskip

\subsection{Definition}\label{mfc.s.def}

%Let $H$ be a Hilbert space.
A {\emdf measurable (functional) calculus}  
on a measurable space $(X,\Sigma)$ is a pair
$(\Phi, H)$ where $H$ is a Hilbert space and 
\[ \Phi: \Meas(X,\Sigma)\to \Clo(H)
\]
is a mapping with the following properties ($f,\: g \in \Meas(X,\Sigma),\: \lambda \in \K$):
\begin{aufziii}
\item[\quad(MFC1)] $\Phi(\car) = \Id$;
\item[\quad(MFC2)] $\Phi(f) + \Phi(g) \subseteq \Phi(f+g)$ and 
$\lambda \Phi(f) \subseteq \Phi(\lambda f)$;
\item[\quad (MFC3)] $\Phi(f)\Phi(g) \subseteq \Phi(fg)$\quad  and 
\[ \dom(\Phi(f)\Phi(g)) = \dom(\Phi(g))\cap \dom(\Phi(fg));
\]
%\item[\quad(MFC4)] $\Phi(f)$ is densely defined 
\item[\quad(MFC4)] $\Phi(f) \in \BL(H)$  and $\Phi(f)^* =
  \Phi(\konj{f})$  if $f$ is bounded\footnote{If $\K = \R$ then 
    $\konj{f} = f$ for all $f\in \Meas(X, \Sigma)$.};
\item[\quad(MFC5)]  If $f_n \to f$ pointwise and boundedly, then
  $\Phi(f_n) \to \Phi(f)$ weakly.  
\end{aufziii}
Property (MFC5) is called the {\emdf weak bp-continuity} of the mapping
$\Phi$. We shall see below, that a measurable functional
calculus is actually (strongly) bp-continuous, i.e., 
one can replace ``weakly'' by ``strongly'' in (MFC5).
(See Theorem \ref{mfc.t.prop1}.f below.)

Given a measurable functional calculus $(\Phi,H)$ we denote by
\[   \bdd(\Phi) := \{ f\in \Meas(X, \Sigma) \suchthat \Phi(f) \in
\BL(H)\}
\]
the set of {\emdf $\Phi$-bounded elements}.

\medskip

\subsection{First Properties}\label{mfc.s.prop1}

In the following, we shall explore and comment on the axioms. 
First of all, (MFC1)--(MFC3) simply say that  a measurable
functional calculus is a proto-calculus  in the terminology of
\cite{Haase2020pre}. As a consequence, a measurable functional
calculus has the properties of every proto-calculus. These account
for a)--c) of the following theorem.

%, and some more. The
% following result collects the more straightforward of these.

\begin{thm}\label{mfc.t.prop1}
Let $\Phi: \Meas(X, \Sigma) \to \Clo(H)$ be a 
measurable functional calculus. Then the following assertions hold 
($f_n,\: f,\:g \in \Meas(X, \Sigma)$, $\lambda \in \C$):
\begin{aufzi}
\item If $\lambda \neq 0$ or $\Phi(f) \in \BL(H)$ then $\Phi(\lambda f) = \lambda \Phi(f)$.
\item If $\Phi(g)\in \BL(H)$ then
\[  \Phi(f) + \Phi(g) = \Phi(f+g)\quad \text{and}\quad \Phi(f)\Phi(g)
= \Phi(fg).
\]
Moreover, $\Phi(g)\Phi(f) \subseteq \Phi(f) \Phi(g)$, i.e., 
$\Phi(f)$ commutes with $\Phi(g)$.

\item If $f\neq 0$ everywhere then $\Phi(f)$ is injective and
  $\Phi(f)^{-1} = \Phi(f^{-1})$.  

\item $\Phi(f)$ is densely defined.

\item If $f$ is bounded, then $\norm{\Phi(f)} \le \norm{f}_\infty$.

\item {\rm (MFC5')}\quad  $\Phi$ is bp-continuous, i.e.:
if $f_n \to f$ pointwise and boundedly, then
$\Phi(f_n) \to \Phi(f)$ strongly. 
\end{aufzi}
\end{thm}

%\noident
%We shall abbreviate the property f) of
%a measurable calculus by {\bf (MFC5')}.

\begin{proof}
Assertion a) and the first part of b) are straightforward consequences
of the axioms (MFC1)--(MFC3) for a proto-calculus, see 
\cite[Thm.{ }2.1.]{Haase2020pre}. The second assertion of b)
follows since $fg = gf$, and hence
\[ \Phi(g)\Phi(f)\subseteq \Phi(gf) = \Phi(fg) = \Phi(f)\Phi(g).
\]
For c) note that if $f\neq 0$ everywhere then $g := 1/f$ satisfies $fg
= \car$, and hence also c) follows from 
general properties of proto-calculi, cf.{ }\cite[Thm.{ }2.1]{Haase2020pre}.

\prfnoi
d) Let $e := (1 + \abs{f})^{-1}$. Then $e$ is bounded and
real-valued, and hence $\Phi(e)$ is self-adjoint by (MFC4). 
By c), $\Phi(e)$ is injective, and hence $\Phi(e)$ has dense
range. But $ef$ is bounded and hence, by b), 
$\Phi(f) \Phi(e) = \Phi(ef)$ is bounded. It follows that 
$\Phi(e)$ maps $H$ into $\dom(\Phi(f))$.

\prfnoi
e)\  This follows from (MFC4) and b)  by a standard argument, which we give for the convenience of the reader. 
Let $f \in \Meas(X,\Sigma)$ with $\abs{f} \le 1$. 
Then, with $g := (1 - \abs{f}^2)^{\frac{1}{2}}$, 
\[  \dprod{(\Id - \Phi(\abs{f}^2))x}{x} =
\dprod{ \Phi(g^2)x}{x} = \dprod{\Phi(g)^*\Phi(g)x}{x} = \norm{\Phi(g)x}^2\ge 0
\]
and hence 
\[ \norm{\Phi(f)x}^2 = \dprod{\Phi(f)^*\Phi(f)x}{x} = 
\dprod{\Phi(\abs{f}^2)x}{x}  \le \dprod{x}{x}= \norm{x}^2
\]
for each $x\in H$. 

\prfnoi
f)\ Suppose that $f_n \to f$ pointwise and boundedly
and let $x \in H$. 
Then, by (MFC5), $\Phi(f_n)x\to \Phi(f)x$ weakly. 
Furthermore,
\[ \norm{\Phi(f_n)x}^2
= \dprod{\Phi(f_n)^*\Phi(f_n)x}{x}
= \dprod{\Phi(\abs{f_n}^2)x}{x} 
\to \dprod{\Phi(\abs{f}^2)x}{x} = \norm{\Phi(f)x}^2,
\]
because  also $\abs{f_n}^2 \to \abs{f}^2$ pointwise and
boundedly. By a standard fact from Hilbert space theory
\cite[Lemma D.18]{EFHN}, it follows that 
$\Phi(f_n)x \to \Phi(f)x$ in norm.
\end{proof}

\begin{rem}
Revisiting the previous proof we see that
a)--c) rest exclusively  on (MFC1)--(MFC3), and 
only f) rests on (MFC5).
\end{rem}

Let us derive some immediate consequences.

\begin{cor}\label{mfc.c.prop1}
Let $\Phi: \Meas(X, \Sigma) \to \Clo(H)$ be a  measurable functional
calculus. Then the following assertions hold 
($f,\:g,\:h \in \Meas(X, \Sigma)$):
\begin{aufzi}
\item If $\abs{f} \le \abs{g}$ then $\dom(\Phi(g)) \subseteq 
\dom(\Phi(f))$ and 
\[  \norm{\Phi(f)x} \le \norm{\Phi(g)x} \qquad (x\in \dom(\Phi(g))).
\]
\item $\dom(\Phi(f)) = \dom(\Phi(\abs{f}))$ and $\norm{\Phi(f)x} =
  \norm{\Phi(\abs{f})x}$ for all $x\in \dom(\Phi(f))$.

\item Let $p(z) = \sum_{j=0}^n a_j z^j \in \K[z]$ be a 
polynomial of degree $n\in \N$. Then 
\[ \Phi(p(f)) = p(\Phi(f)) = \sum_{j=0}^n a_j \Phi(f)^j.
\]
In particular, $\dom( \Phi(p(f))) = \dom( \Phi(f)^n)$.

%\item If $\Phi(h)$ is bounded, then $\Phi(h)\Phi(f)\subseteq
%%\Phi(f)\Phi(h)$, i.e.,
%\[  (x,y) \in \Phi(f) \quad \dann\quad (\Phi(h)x, \Phi(h)y) \in
%\Phi(f)\qquad (x,y\in H).
%\]
\end{aufzi}
\end{cor}

\begin{proof}
a) If $\abs{f}\le \abs{g}$ then we can write $f = hg$, where $h$ is
the function
\[  h := \begin{cases} \frac{f}{g} & \text{on}\quad \set{g\neq 0}\\
0 & \text{on}\quad \set{g= 0}.
\end{cases}
\]
Then $\abs{h}\le 1$ and hence $\Phi(h)$ is bounded with
$\norm{\Phi(h)} \le 1$. From (MFC3) it follows that 
\[ \Phi(h)\Phi(g) \subseteq \Phi(f)
\]
with $\dom(\Phi(g)) \subseteq \Phi(f)$. Furthermore, if $x\in
\dom(\Phi(g))$
we obtain
\[ \norm{\Phi(f)x} = \norm{\Phi(h)\Phi(g)x} \le \norm{\Phi(g)x}
\]
as claimed.

\prfnoi
b) follows from a).

\prfnoi
c)\  For $n \ge 2$ write 
$f^n = f^n \car_{\set{\abs{f}\le 1}} + f^n \car_{\set{\abs{f} >
    1}}$. Since the first summand is bounded, one has
\[  \dom(\Phi(f^n)) = \dom( \Phi( f^n \car_{\set{\abs{f} > 1}})) 
\subseteq \dom( \Phi( f^{n{-}1} \car_{\set{\abs{f} > 1}}))
= \dom( \Phi(f^{n{-}1}))
\]
by a). It follows that 
\[ \Phi(f) \Phi(f^{n{-}1}) =  \Phi(f^n).
\]
by (MFC3). By induction, we obtain
\[ \Phi(f^n) = \Phi(f)^n \quad (n \ge 1).
\]
Now take $p \in \K[z]$ as in the hypothesis. Since $\deg(p) = n$, we
have $a_n \neq 0$ and there are numbers $0 < a < b$ and $c> 0$
such that 
\[      a\abs{z}^n \le \abs{p(z)} \le b \abs{z}^n \qquad (\abs{z} \ge
c).
\]
Similarly as above, multiplying with $\car_{\set{\abs{f}\ge c}}$ shows
that 
\[ \dom( \Phi(p(f))) = \dom(\Phi(f^n)) =
\dom(\Phi(f)^n).
\] 
Since
\[ \sum_{j=0}^n a_j \Phi(f)^j \subseteq \Phi(p(f)).
\]
by (MFC2), the assertion is proved. 
\end{proof}

So far, we have used (MFC5) only to establish the strengthening
(MFC5').  We shall explore further consequences of (MFC5) in the
following section.

\medskip

\subsection{Approximations of the Identity and Further Properties}\label{mfc.s.prop2}

Let $(\Phi, H)$ be a measurable functional calculus 
on the measurable space $(X, \Sigma)$.
An {\emdf approximate identity} in $\Meas(X, \Sigma)$ is a
sequence $(e_n)_n$ of bounded measurable functions
such that $e_n \to \car$ pointwise and boundedly. 
It then follows from (MFC5') (see Theorem \ref{mfc.t.prop1}.f)
that $\Phi(e_n) \to \Id$
strongly on $H$.

Such approximate identities abound. For instance,
given $f\in \Meas(X, \Sigma)$ the
sequences of functions
\[ e_n := \frac{n}{n + \abs{f}} \quad \text{and}
\quad \widetilde{e}_n := \car_{\set{\abs{f}\le n}}\qquad (n \in \N)
\]
are both approximate identities. Furthermore, as 
$e_n^{-1} = 1 + \frac{1}{n} \abs{f}$, one has
\[ \Phi(e_n)^{-1} = 
\Phi(e_n^{-1}) = \Id + \tfrac{1}{n} \Phi(\abs{f}).
\]
This yields
\[ \dom(\Phi(f)) = \dom(\Phi(\abs{f}))
= \dom(\Phi(e_n)^{-1}) = \ran(\Phi(e_n))
\]
for each $n\in \N$. It follows once more that 
$\Phi(f)$ must be densely defined. 
But more is true.

\begin{thm}\label{mfc.t.prop2}
Let $(\Phi, H)$ be a measurable functional calculus 
on the measurable space $(X, \Sigma)$.
Then the following assertions hold
($f,\:g\in \Meas(X, \Sigma)$):
\begin{aufzi}
\item $\dom(\Phi(f)) \cap \dom(\Phi(g))$ is a 
core for $\Phi(g)$. 

\item $\cls{\Phi(f)+\Phi(g)} = \Phi(f+g)$\quad and\quad
$\cls{\Phi(f)\Phi(g)} = \Phi(fg)$.

\item $\Phi(f)^* = \Phi(\konj{f})$.

\item $\Phi(f)$ is normal and $\Phi(\konj{f}) \Phi(f)
= \Phi(\abs{f}^2)$.

\item If $f$ is real-valued then $\Phi(f)$ is self-adjoint. 
Moreover, if $f$ and $g$ are real-valued and $f\le g$, then
\[  \dprod{\Phi(f)x}{x} \le \dprod{\Phi(g)x}{x}
\]
for all $x\in \dom(\Phi(f)) \cap \dom(\Phi(g))$.

\item The set $\bdd(\Phi)$ of $\Phi$-bounded elements is a unital
  $*$-subalgebra of $\Meas(X, \Sigma)$ and 
\[ \Phi: \bdd(\Phi) \to \BL(H)
\]
is a unital $*$-homomorphism.  Moreover, the following
generalization of {\rm (MFC5)} holds: if $(f_n)_n$ is a sequence
in $\bdd(\Phi)$ such that $f_n \to f$ pointwise and 
$\sup_n \norm{\Phi(f_n)} < \infty$, then  $f\in \bdd(\Phi)$ and
$\Phi(f_n) \to \Phi(f)$ strongly on $H$.  
\end{aufzi}
\end{thm}

\begin{proof}
a)\  Define $e_n := n(n + \abs{f})^{-1}$ as in the remark above.
Suppose that $x,y\in H$ are such that $\Phi(g)x = y$.
Let $x_n := \Phi(e_n)x$ and $y_n:=\Phi(e_n)y$. 
Then  $\Phi(g)x_n = y_n$ (by Theorem  \ref{mfc.t.prop1}.b)
and $x_n \in \dom(\Phi(f))$ (by definition  of $e_n$).
Since $(x_n, y_n) \to (x,y)$ as
$n \to \infty$, the claim follows.

\prfnoi
b)\ To prove the first identity, 
let $h := \abs{f} + \abs{g}$. Then
$\dom(\Phi(h)) \subseteq \dom(\Phi(f)) \cap 
\dom(\Phi(g))$ by Corollary \ref{mfc.c.prop1}. On the
other hand, 
\[ \dom(\Phi(h)) = \dom(\Phi(h)) \cap \dom(\Phi(f+g))
\]
is a core for $\Phi(f+g)$ by a). The first identity follows.

For the proof of the second identity, let
$h := \abs{g} + \abs{fg}$. Then (by (MFC3))
\[ \dom(\Phi(h)) \subseteq \dom(\Phi(g)) \cap \dom(\Phi(fg))
= \dom(\Phi(f)\Phi(g)).
\]
By a), $\dom(\Phi(h))$ is a core for 
$\Phi(fg)$. Hence the second identity. 

\prfnoi
c)\ Let $e \in \BMb(X, \Sigma)$ real-valued and such that $ef$ is bounded. We then have $\Phi(e) = \Phi(e)^*$ and hence
\[ \Phi(e)  \Phi(f)^* 
\subseteq  \cls{ \Phi(e) \Phi(f)^*} 
= (\Phi(f) \Phi(e))^* = \Phi(fe)^* = \Phi(\konj{f}e) = 
\Phi(\konj{f}) \Phi(e)
\]
by \cite[Prop.C.2.1.k]{HaaseFC}.  By taking $e= e_n 
= n (n + \abs{f})^{-1}$ we conclude from $\Phi(\konj{f})$ being closed that 
$\Phi(f)^* \subseteq \Phi(\konj{f})$. On the other hand, 
since $\Phi(ef) = \cls{\Phi(e)\Phi(f)}$ by b) and 
again by \cite[Prop.C.2.1.k]{HaaseFC} we obtain
\[ \Phi(e) \Phi(\konj{f})  
\subseteq \Phi(e \konj{f}) = \Phi(ef)^* =
( \Phi(e) \Phi(f))^* = \Phi(f)^* \Phi(e). 
\]
With the same argument as before, now employing that
$\Phi(f)^*$ is closed, we obtain
$\Phi(\konj{f}) \subseteq \Phi(f)^*$.

\prfnoi
d) It is clear by (MFC3) that $\Phi(\konj{f}) \Phi(f) \subseteq 
\Phi(\abs{f}^2)$. Hence, by c),  
\[ \Id + \Phi(f)^* \Phi(f) \subseteq \Phi(1 + \abs{f}^2).
\]
By Theorem \ref{mfc.t.prop1}, 
the operator on the right-hand side is 
injective while the 
operator on the left-hand side is surjective.
(This is standard
Hilbert space operator theory, see \cite[Thm. 13.13]{RudinFA}.)
Hence, these operators must coincide. Normality of $\Phi(f)$ follows readily. 
 
\prfnoi
e) The first assertion follows from c). For the second,
let $ u := \sqrt{g-f}$ and $x\in \dom(\Phi(f)) \cap
\dom(\Phi(g))$. 
Then $x\in \dom(\Phi(g-f)) = \dom(\Phi(u)^2)$ and since
$\Phi(u)$ is self-adjoint, 
\[ \dprod{\Phi(g)x}{x} - \dprod{\Phi(f)x}{x} 
= \dprod{ \Phi(g-f)x}{x} = \dprod{\Phi(u)^2x}{x}
= \norm{\Phi(u)x}^2 \ge 0.
\]

\prfnoi
f)\ The first assertion follows readily 
from c) above and from Theorem \ref{mfc.t.prop1}.
For the second assertion, suppose that 
$(f_n)_n$ is a sequence in $\Meas(X, \Sigma)$
such that $f_n \to f$ pointwise and $\sup_{n}
\norm{\Phi(f_n)} < \infty$. Write
\[  \frac{f_n}{1 + \abs{f}} = f_n\Bigl( 
\frac{1}{1+ \abs{f}} - \frac{1}{1+\abs{f_n}}\Bigr)
+ \frac{f_n}{1 + \abs{f_n}}.
\]
By Theorem \ref{mfc.t.prop1}, it follows that
\[\Phi(f_n)\Phi\Bigl(\frac{1}{1+\abs{f}}\Bigr)=
 \Phi\Bigl( \frac{f_n}{1 + \abs{f}}\Bigr)
 = \Phi(f_n) \Phi\Bigl( 
\frac{1}{1+ \abs{f}} - \frac{1}{1+\abs{f_n}}\Bigr)
+ \Phi\Bigl(\frac{f_n}{1 + \abs{f_n}}\Bigr).
\]
By (MFC5') and the uniform  boundedness of
the operators $\Phi(f_n)$, we obtain
\[ \Phi(f_n)\Phi\Bigl(\frac{1}{1+\abs{f}}\Bigr) 
\to \Phi\Bigl(\frac{f}{1 + \abs{f}}\Bigr)
= \Phi(f) \Phi\Bigl(\frac{1}{1 + \abs{f}}\Bigr)
\]
strongly on $H$. Hence, for all $x \in \dom(\Phi(f))$ 
one has $\Phi(f_n)x \to \Phi(f)x$. The
uniform boundedness assumption implies 
that
$\Phi(f)$ is norm bounded on its domain, and since
it is a closed operator, $\Phi(f) \in \BL(H)$. Again
from the uniform boundedness it follows that
$\Phi(f_n)x\to \Phi(f)x$ for all $x\in H$.
\end{proof}

\begin{rem}
By essentially the same arguments, one can
prove the following generalization of a) and b) from Theorem
\ref{mfc.t.prop2}: {\em 
Let $(\Phi, H)$ be a measurable functional calculus 
on the measurable space $(X, \Sigma)$ and let
$f_1, \dots, f_d \in \Meas(X,\Sigma)$. Then 
\begin{aufzi}
\item the space $\bigcap_{j=1}^d \dom(\Phi(f_j))$ is a
core for each operator $\Phi(f_j)$, $j=1, \dots, d$, 
\item $ 
\cls{\Phi(f_1)+ \dots +\Phi(f_d)} = \Phi(f_1 + \dots + f_d)$,
 and
\item $
\cls{\Phi(f_1)\cdots \Phi(f_d)} = \Phi(f_1 \cdots f_d)$.
\end{aufzi}
}%
\end{rem}

\medskip

{\bf From now on, we shall use the properties
of measurable functional calculi 
expressed in Theorems \ref{mfc.t.prop1}
and \ref{mfc.t.prop2} and Corollary \ref{mfc.c.prop1}
without explicit reference.}

\medskip

\subsection{Determination by Bounded Functions}\label{mfc.s.det}

Let $(\Phi, H)$ be a measurable functional calculus 
on the measurable space $(X, \Sigma)$ and let
$f\in \Meas(X, \Sigma)$ be arbitrary. Define
\[  e :=  \frac{1}{1 + \abs{f}}
\]
Then $e, ef$ are bounded functions and hence $\Phi(e), \Phi(ef)$ are
bounded operators. As $e$ is nowhere
zero, $\Phi(e)$ is injective. (In the terminology of \cite{HaaseFC}, this means
that $e$ is a {\emdf regularizer} of $f$.)
As a consequence, we obtain 
\[ \Phi(f) = \Phi(e^{-1} ef) = \Phi(e^{-1}) \Phi(ef) 
= \Phi(e)^{-1}\Phi(ef).
\]
The identity $\Phi(f) = \Phi(e)^{-1}\Phi(ef)$ thus established can be
rephrased by means of  the equivalence
\[ \Phi(f)x = y \quad \iff\quad  \Phi(ef)x = \Phi(e)y \qquad (x,\:y\in H).
\]
We realize that $\Phi$ is completely determined by its restriction
to $\BMb(X, \Sigma)$. In particular, a measurable functional calculus
is a calculus in the sense of \cite{Haase2020pre}, and
$\BMb(X,\Sigma)$ is an ``algebraic core'' in the terminology introduced
there.

\medskip
\subsection{Restriction to Subspaces}\label{mfc.s.prt}

Suppose that $(\Phi, H)$ is a measurable calculus on $(X, \Sigma)$,
and that $K \subseteq H$ is a closed subspace of $H$, with
$P\in \BL(H)$ being the orthogonal projection onto $K$. For each
$f\in \Meas(X, 
\Sigma)$ we let $\Phi_K(f)$ be the {\em part} of $\Phi(f)$ in $K$,
that is, the operator
\[ \Phi_K(f) := \Phi(f) \cap (K \oplus K).
\]
The mapping $\Phi_K: \Meas(X, \Sigma) \to \Clo(K)$ is called
the {\emdf restriction} of $\Phi$ to $K$.

\begin{lem}\label{mfc.l.prt}
In the situation just described, the following assertions
hold:
\begin{aufzi}
\item $\Phi_K$ satisfies {\rm (MFC1)---(MFC3)}.
\item The following are equivalent:
\begin{aufzii}
\item $\Phi_K$ is a measurable functional calculus.
\item  $\Phi_K(f) \in \BL(K)$ for each $f\in \BMb(X, \Sigma)$.
\item $K$ is invariant under each $\Phi(f)$, $f\in \BMb(X,
  \Sigma)$.
\item $\Phi(f)P = P\Phi(f)$ for each  $f\in \BMb(X, \Sigma)$.
\end{aufzii}
\end{aufzi}
\end{lem}

\begin{proof}
a) is straightforward to verify.

\prfnoi
b)\ It is clear that (i) implies (ii). If (ii) holds and $f\in \BMb(X,
\Sigma)$, then $\Phi_K(f) = \Phi(f) \cap (K \oplus K)$ is fully
defined. But that means that $\Phi(f)$ must map $K$ into $K$. This
yields (iii). Suppose that (iii) holds and let $f\in\BMb(X, \Sigma)$.
Then  $\Phi(f)P = P\Phi(f)P$. Since $K$ is also invariant under
$\Phi(\konj{f}) = \Phi(f)^*$, one also has
\[   \Phi(f)(\Id - P) = (\Id -P) \Phi(f) (\Id - P)
\]
and combining both identities yields (iv).

\prfnoi
The implications (iv)$\dann$(iii)$\dann$(ii) are trivial.
If (ii) holds, then one obviously has (MFC4) and (MFC5). 
Hence, by a), (i) follows. 
\end{proof}

\medskip

\subsection{Pull-back and Push-Forward 
of a Measurable Calculus}\label{mfc.s.pull}

Suppose that $(\Phi, H)$ is a measurable calculus on $(X, \Sigma)$,
and $U: H \to K$ is an isometric isomorphism of Hilbert spaces $H$ and
$K$.
Then a measurable calculus $(\Psi, K)$ may be defined  by 
\[ \Psi(f) := U \Phi(f) U^{-1} \qquad (f\in \Meas(X, \Sigma)).
\]
(This is easily checked.)

\medskip
In contrast to the above situation, in which the measurable space
is kept and the Hilbert space is changed, 
one may transfer a measurable calculus to
a different measurable space in the following way.

\begin{prop}\label{mfc.p.pullback}
Let $(X, \Sigma_X)$ and $(Y, \Sigma_Y)$ be measurable spaces,
let $(\Phi, H)$  be a 
measurable functional calculus on $(X, \Sigma_X)$
and let $T: \Meas(Y, \Sigma_Y) \to \Meas(X, \Sigma_X)$
be a $*$-homomorphism with the property that
if $f_n \to f$ pointwise and boundedly, then
$Tf_n \to Tf$ pointwise and boundedly. Then the 
mapping
\[ (T^*\Phi): \Meas(Y, \Sigma_Y) \to \Clo(H),\qquad
(T^*\Phi)(g) := \Phi(Tg),
\]
is a measurable functional calculus on $(Y, \Sigma_Y)$.
\end{prop}

\begin{proof}
Straightforward.
\end{proof}

\noindent
The new functional calculus $T^*\Phi$ is called the 
{\emdf pull-back} of $\Phi$ {\emdf along} $T$.

A particular instance of a pull-back occurs
in the case of a measurable mapping 
$\vphi: X \to Y$.  The induced ``Koopman operator''
\[\ T_\vphi:  \Meas(Y, \Sigma_Y) \to \Meas(X, \Sigma_X),
\qquad T_\vphi(g) := g \nach \vphi
\]
satisfies the hypothesis of Proposition \ref{mfc.p.pullback}. Hence, 
its pull-back is 
\[ \Phi^\vphi := (T_\vphi^*\Phi) : \Meas(Y\Sigma_Y) \to \Clo(H),\qquad 
\Phi^\vphi(f) := \Phi(f\nach \vphi).
\]

\begin{cor}\label{mfc.p.push}
In the situation described above, the mapping 
$\Phi^\vphi: \Meas(Y, \Sigma_Y) \to \Clo(H)$ is a measurable
functional calculus.
\end{cor}

\noindent
The calculus $(\Phi^\vphi, H)$ is called the {\emdf push-forward} 
of $\Phi$ {\emdf along $\vphi$}.

\medskip

This construction applies in particular in the case that $(Y,\Sigma_Y)
= (\K, \Borel(\K))$ or, more generally, $(Y,\Sigma_Y)
= (\K^d, \Borel(\K^d))$. To wit, each tuple 
$\vphi := (\vphi_1, \dots, \vphi_d)$ of measurable scalar functions
induces a measurable calculus on $\Meas(\K^d,
\Borel(\K^d))$. We shall see below in Proposition  \ref{uni.p.uni-mult} 
that this calculus
does only depend on the tuple of operators $\Phi(\vphi_1), \dots, 
\Phi(\vphi_d)$.

\medskip

\subsection{Multiplication Operators}\label{mfc.s.mul}

Let $\prO = (\Omega, \calF,\mu)$ be any measure space.
% and $1 \le p \le \infty$. 
For a measurable function $a: \Omega \to \K$ we define
the corresponding {\emdf multiplication operator} $M_a$ on
$H:= \Ell{2}(\prO)$ by 
\[ M_a x = y \,\,\defiff\,\, ax = y,
\]
where on the right-hand side we mean equality of equivalence classes,
i.e., almost everywhere equality of representatives. In other words,
\[ 
M_ax := af \quad \text{for}\quad f\in \dom(M_a) := \{f\in \Ell{2}(\prO) \suchthat af\in \Ell{2}(\prO)\}.
\]
It is obvious that $M_a$ depends only on the equivalence class of the
function $a$
modulo equality almost everywhere. We shall freely make use of this
observation in the following and form operators $M_a$ also in the
case when  $a$ is merely essentially measurable. 

%The following theorem is standard, and we include a proof only for convenience.

\begin{thm}\label{mfc.t.mul}
Let $\prO = (\Omega, \calF,\mu)$ be a measure space and define
\[ \Phi(a) := M_a \qquad (a\in \Meas(\Omega, \calF)).
\]
Then $\Phi$ is a measurable functional calculus.
\end{thm}

\begin{proof}
This is straighforward. 
%(Note
%that for each $a\in \Meas(\Omega, \calF)$, each $n \in \N$ and
%each $f\in \Ell{2}(\prO)$ one has $\car_{\set{\abs{a}\le n}} f\in
%\dom(M_a)$.)
\end{proof}

As described in the previous section, 
the measurable calculus described above generates a wealth of
related measurable calculi as push-forwards. Let, as above, 
$\prO = (\Omega,\calF,\mu)$ be a measure space, and let 
$(X, \Sigma)$ be any measurable space and 
\[ a: \Omega \to X
\]
a measurable (or just essentially measurable) function. For a
measurable function $f\in \Meas(X, \Sigma)$ define
\[ \Phi(f) := M_{f\nach a}, 
\]
which is a closed operator on $H := \Ell{2}(\prO)$. By Corollary \ref{mfc.p.push},
the mapping $\Phi: \Meas(X, \Sigma) \to \Clo(H)$ is a measurable
functional
calculus.

\section{Projection-Valued Measures and Null Sets}\label{s.pvm}

If $(\Phi, H)$ is a measurable functional
calculus on a measurable space $(X, \Sigma)$, then 
the mapping
\[ \Ee_\Phi: \Sigma \to \BL(H),\qquad 
\Ee_\Phi(B) := \Phi(\car_B) \in \BL(H) 
\qquad (B \in \Sigma)
\]
is a {\emdf projection-valued measure}. This means that 
$\Ee := \Ee_\Phi$  has the following, easy-to-check properties:
\begin{aufziii}
\item[\quad(PVM1)]  $\Ee(B)$ is an orthogonal projection on $H$
for each $B \in \Sigma$.
\item[\quad(PVM2)]  $\Ee(X) = \Id$.
\item[\quad(PVM3)]  If $B = \bigsqcup_{n=1}^\infty B_n$ with all $B_n \in \Sigma$ then 
$\sum_{n=1}^\infty \Ee(B_n) = \Ee(B)$ in the strong 
(equivalently: weak) operator topology.
\end{aufziii}
(The equivalence of convergence in weak and strong
operator topology in 3) is shown similarly as f) in 
Theorem \ref{mfc.t.prop1}.) 
A  projection-valued
measure is nothing but a {\emdf resolution of the identity}
in the terminology of Rudin \cite[12.17]{RudinFA}.

Several concepts and results treated from now on actually depend
only on the properties of the projection-valued measure. 
However,  it is well known that associated with each projection-valued
measure $\Ee$ on $(X, \Sigma)$ there exists a (unique)
measurable functional calculus $\Phi_\Ee$ such that 
$\Ee= \Ee_{\Phi_\Ee}$ (see Theorem  \ref{con.t.ext-pvm} below). It is therefore
no loss of generality when we treat said concepts and results
in the framework of measurable calculi.

\medskip

\subsection{Null Sets}\label{pvm.s.null}

Let $(\Phi,H)$ be a fixed measurable functional
calculus on $(X,\Sigma)$. Then a set $B \in \Sigma$
is called a {\emdf $\Phi$-null set} if 
$\Phi(\car_B) = 0$. The set
\[ \calN_\Phi := \{ B \in \Sigma \suchthat \Phi(\car_B)=0\}
\]
of $\Phi$-null sets
is a $\sigma$-ideal of $\Sigma$. (This is a simple exercise.) 
Similarly to 
standard measure theory, we say that
something happens {\emdf $\Phi$-almost everywhere} if it fails to  happen
at most on a $\Phi$-null set. For instance, the assertion
 ``$f=g$ $\Phi$-almost everywhere'' for two functions 
$f,\:g\in \Meas(X, \Sigma)$
means just that  $\set{f\neq g} \in \calN_\Phi$.

\begin{lem}\label{pvm.l.null}
Let $(\Phi,H)$ be a measurable functional
calculus on $(X,\Sigma)$ and let  $f,\: g\in \Meas(X,\Sigma)$.
Then the following assertions  hold:
\begin{aufzi} 
\item $\ker( \Phi(f)) = \ker \bigl(\Phi(\car_{\set{f\neq 0}})\bigr)$.

\item  $\Phi(f) = 0 \quad  \iff\quad f = 0$\quad $\Phi$-almost everywhere.

\item $\Phi(f) = \Phi(g)\quad \iff\quad f= g$\quad$\Phi$-almost
  everywhere.

\item $\Phi(f) \,\,\text{is injective} \quad  \iff\quad f \neq 0$\quad 
$\Phi$-almost everywhere.
\end{aufzi}
\end{lem}

\begin{proof}
a)\ Since $f = f \car_{\set{f\neq 0}}$ one has $\Phi(f) =
\Phi(f)\Phi(\car_{\set{f\neq 0}})$. This yields the inclusion
``$\supseteq$''. Next, define $g := f^{-1}
\car_{\set{f\neq0}}$. 
Then $gf = \car_{\set{f\neq 0}}$ and hence $\Phi(g)\Phi(f)\subseteq
\Phi(\car_{\set{f\neq 0}})$. This yields the inclusion
``$\subseteq$''.

\prfnoi
b)\  By a), $\Phi(f) = 0$ if and only if $\Phi( \car_{\set{f\neq 0}})
= 0$, if and only if $f= 0$ $\Phi$-almost
everywhere. 

\prfnoi
c)\ If $f= g$\, $\Phi$-almost everywhere, then $f- g = 0$
$\Phi$-almost everywhere and hence, by b), $\Phi(f-g) = 0$. Since
$f = g + (f-g)$, it follows by general functional calculus 
properties that 
$\Phi(f) = \Phi(g) + \Phi(f-g) = \Phi(g)$. 

Conversely, suppose that $\Phi(f) = \Phi(g)$. Abbreviate $A_n := \set{
  \abs{f} + \abs{g}\le n}$ for $n \in \N$. 
Then 
\[ \Phi( (f-g)\car_{A_n}) = \Phi(f\car_{A_n} - g\car_{A_n}) 
=  \Phi(f)\Phi(\car_{A_n}) - \Phi(g)\Phi(\car_{A_n}) = 0
\]
and hence, by b), $\set{f\neq g} \cap A_n \in \calN_\Phi$. Since
$\calN_\Phi$
is a $\sigma$-ideal, $f = g$ $\Phi$-almost everywhere. 

\prfnoi
d)\ By a), $\Phi(f)$ is injective if and only if
$\Phi(\car_{\set{f\neq 0}})$ 
is injective, if and only if $\Phi(\car_{\set{f\neq 0}}) = \Id$ 
(since it is an orthogonal projection),
if and only if $\Phi(\car_{\set{f= 0}})  = \Id - \Phi(
\car_{\set{f\neq 0}}) = 0$. 
\end{proof}

\medskip

\subsection{Concentration}\label{mfc.s.concen}

Let $(\Phi, H)$ be a measurable calculus on $(X, \Sigma)$. 
We say that $\Phi$ is {\emdf concentrated} on 
a set $Y\in \Sigma$ if $Y^c$ is a $\Phi$-null set. 
For $Y \subseteq X$  denote by $\Sigma_Y$
the {\emdf trace $\sigma$-algebra}
\[ \Sigma_Y := \{ Y \cap B \suchthat B \in \Sigma\}.
\]
If $\Phi$ is concentrated on $Y \in \Sigma$
then one can induce a measurable calculus on $(Y, \Sigma_Y)$
by defining 
\[ \Phi_Y(f) := \Phi(f^Y) \qquad (f\in \Meas(Y,\Sigma_Y)),
\]
where
\[  f^Y :=  \begin{cases} f & \text{on $Y$}\\
0 & \text{on $Y^c$}.
\end{cases}
\] 
Axioms (MFC2)--(MFC5) are immediate, and Axiom
(MFC1) holds since $\Phi$ is concentrated on $Y$.

\medskip

Conversely, if $(\Phi,H)$ is a measurable functional
calculus on $(Y, \Sigma_Y)$ then by 
\[ \Phi_X(f) := \Phi(f\res{Y})\qquad (f\in 
\Meas(X,\Sigma))
\]
one obtains a measurable functional calculus 
$(\Phi_X,H)$
on $(X, \Sigma)$ concentrated on $Y$. 
(This calculus is nothing but the push-forward
of $\Phi$ along the inclusion mapping.)
In this way, 
for a measurable set $Y \subseteq X$
a one-to-one correspondence is established between
measurable functional calculi on $(Y,\Sigma_Y)$ 
on one side
and measurable functional calculi on $(X, \Sigma)$
concentrated on $Y$ on the other.

\medskip
\subsection{Support of a Borel Calculus}\label{pvm.s.supp}

A measurable functional calculus $(\Phi, H)$ on $(X, \Sigma)$ is
called a {\emdf Borel calculus} if $X$ carries a topology and $\Sigma= 
\Borel(X)$ is the Borel $\sigma$-algebra, i.e., the smallest
$\sigma$-algebra on $X$ that contains all open sets.

In this case we call the closed subset
\[ \supp(\Phi) := X \ohne \bigcup \{ U  \suchthat 
\text{$U \in \calN_\Phi$ and $U$ is open in $X$}\}
\]
the {\emdf support} of $\Phi$. A point $x\in X$ is contained in
$\supp(\Phi)$ if and only if no open neighbourhood 
of $x$ is a $\Phi$-null set. The next result is obvious.

\begin{prop}\label{pvm.p.supp}
Let $(\Phi, H)$ be a Borel calculus on a second countable 
topological space $X$. Then $\Phi$ is concentrated on $\supp(\Phi)$. 
\end{prop}

\section{Spectral Theory}\label{s.spc}

In this section we shall see that a measurable calculus $(\Phi, H)$ contains
the complete information about the spectrum of each operator $\Phi(f)$.
To this end, define the $\Phi$-{\emdf essential range} of $f\in \Meas(X,\Sigma)$ 
by 
\beq\label{spc.eq.essran}
 \essran_\Phi(f) := \{ \lambda\in \K \suchthat  \forall\, \veps > 0:
\set{ \abs{f -
    \lambda} \le \veps} \notin \calN_\Phi\}.
\eeq
Then we have the following important result.

\begin{thm}\label{spc.t.main}
Let $(\Phi,H)$ be a measurable functional
calculus on $(X,\Sigma)$, let  $f \in \Meas(X,\Sigma)$, $c\ge 0$ and
$\lambda\in \K$. 
Then the following assertions  hold:
\begin{aufzi} 
\item $\spec(\Phi(f)) = \Aspec(\Phi(f)) = \essran_\Phi(f)$.

\item $f \in \essran_\Phi(f)$\,\, $\Phi$-almost everywhere.

\item $\Phi(f)\in \BL(H),\, \norm{\Phi(f)}\le c\quad  
\iff\quad \abs{f} \leq c$\quad 
$\Phi$-almost everywhere.

\item $\Phi(f)$ is self-adjoint$\quad\iff\quad$ $f\in \R$\quad $\Phi$-almost everywhere.

\item $\lambda$ is an eigenvalue of $\Phi(f)$ iff $\set{f= \lambda}$
  is not a $\Phi$-null set. And  $\Phi(\car_{\set{f= \lambda}})$ 
is the projection onto 
the eigenspace $\ker(\lambda - \Phi(f))$. 

\end{aufzi}
%\end{align*} 
\end{thm}

\begin{proof}
a)\  Passing to $\lambda - f$ if necessary we only need to show that
\[ 0 \in \spec(\Phi(f))\quad \Dann\quad 0 \in \essran_\Phi(f) 
\quad \Dann\quad 0 \in \Aspec(\Phi(f)).
\]
If $0 \notin\essran(f)$ then there is $\veps > 0$ such that
$\set{ \abs{f}\le \veps} \in \calN_\Phi$. Define
%Hence by d),  
% $\Phi(f)$ is injective, and
%$M_f^{-1} = M_{g}$, where $g$ is defined as
\[  g = \begin{cases}  f & \text{on $\set{\abs{f}\ge \veps}$}\\
                       \veps      & \text{else}.
\end{cases}
\]
Then $g \neq 0$ everywhere and $g^{-1}$ is bounded, so 
$\Phi(g)$ is  invertible. Since
$g = f$ $\Phi$-almost everywhere and hence $\Phi(f) = \Phi(g)$ by Lemma \ref{pvm.l.null},
we obtain $0 \in \resol(\Phi(f))$. 

Suppose now that $0 \in \essran_\Phi(f)$ and fix $n \in \N$. 
Then $A_n := \set{ \abs{f} \le \frac{1}{n}}$ is not $\Phi$-null, hence there
is a unit vector $x_n \in H$ with 
\[ x_n = \Phi(\car_{A_n}) x_n.
\]
Then
\[  \norm{\Phi(f)x_n} = \norm{\Phi(f) \Phi(\car_{A_n}) x_n} = 
\norm{\Phi(f \car_{A_n}) x_n} \le  \norm{f\car_{A_n}}_\infty \norm{x_n}
\le \frac{1}{n}.
\]
It follows that $(x_n)_n$ is an approximate eigenvector for $0$, and therefore
$0 \in \Aspec(\Phi(f))$ as claimed.

\prfnoi
b)\ Abbreviate  $M := \essran_\Phi(f)$. 
For each $\lambda \in \K\ohne M$ there is $\veps_\lambda > 0$ such that
$\set{f \in \Ball(\lambda, \veps_\lambda)}$ is a $\Phi$-null
set. Since countably many $\Ball(\lambda, \veps_\lambda)$ suffice to cover
$\K\ohne M$ and  $\calN_\Phi$ is a $\sigma$-ideal, it follows that 
$\set{f \notin M}$ is a $\Phi$-null set, and hence that $f \in M$ 
$\Phi$-almost everywhere.

\prfnoi
c)\ If $\abs{f}\le c$ $\Phi$-almost everywhere then 
$f = f \car_{\set{\abs{f}\le c}}$ $\Phi$-almost everywhere. 
Hence by Lemma \ref{pvm.l.null},   
\[ \norm{\Phi(f)} = \norm{\Phi(f\car_{\set{\abs{f}\le c}})} 
\le \norm{ f \car_{\set{\abs{f}\le c}}}_\infty \le c.
\]
Conversely, suppose that $\norm{\Phi(f)} \le c$. Then 
$\essran_\Phi(f) = \spec(\Phi(f)) \subseteq \Ball[0,c]$, and 
therefore
$\abs{f}\le c$ $\Phi$-almost everywhere, by b).

\prfnoi
d)\ If $f\in \R$ $\Phi$-almost everywhere, then $f = \konj{f}$
$\Phi$-almost everywhere, which implies (by Lemma \ref{pvm.l.null})
\[ \Phi(f) = \Phi(\konj{f}) = \Phi(f)^*.
\]
Conversely, if $\Phi(f)$ is self-adjoint, then $\spec(\Phi(f))
\subseteq \R$, and hence $f\in \R$ $\Phi$-almost everywhere, by b).

\prfnoi
e)\ Without loss of generality $\lambda = 0$. Let $P :=
\Phi(\car_{\set{f=0}})$. Then
\[ \ran(P) = \ker(\Id - P) = \ker(\Phi(\car_{\set{f\neq 0}})) 
= \ker(\Phi(f))
\]
by a) of Lemma \ref{pvm.l.null}. 
\end{proof}

\begin{cor}\label{spc.c.spc-supp}
Let $(\Phi, H)$ be a measurable functional calculus on $(X, \Sigma)$
and $f\in \Meas(X,\Sigma)$. Then 
\[  \Pspec(\Phi(f))  \subseteq f(X) \quad \text{and}\quad 
\spec(\Phi(f)) = \supp(\Phi^f) \subseteq \cls{f(X)},
\]
where $\Phi^f$ is the push-forward calculus on $\K$ defined by 
$\Phi^f(g) = \Phi(g\nach f)$. A fortiori, $\Phi^f$ is concentrated
on $\spec(\Phi(f))$.
\end{cor}

\begin{proof}
If $\lambda \in \Pspec(\Phi(f)$ then $\set{f= \lambda}$ is not
$\Phi$-null, and
in particular $\set{f= \lambda} \neq \leer$. Hence $\Pspec(\Phi(f))
\subseteq  f(X)$.

\prfnoi
Since $\spec(\Phi(f)) = \essran_\Phi(f)$, we have 
$\lambda \notin \spec(\Phi(f))$ iff there is
$\veps > 0$ such that $\set{\abs{f- \lambda} < \veps}$ is $\Phi$-null,
which is equivalent to say that the ball $\Ball(\lambda, \veps)$ is 
$\Phi^f$-null. This shows that $\spec(\Phi(f)) = \supp(\Phi^f)$. 
The inclusion $\essran_\Phi(f) \subseteq \cls{f(X)}$ is clear.
The last assertion then follows from Proposition \ref{pvm.p.supp}.
\end{proof}

In the special case $X= \K$ one can apply the previous result
to the mapping $f = \bfz := (z \mapsto z)$.

\begin{cor}\label{spc.c.single}
Let $(\Phi,H)$ be a Borel functional calculus on $\K$, and let
$A := \Phi(\bfz)$. Then $\Phi$ is
concentrated on $\spec(A)$ but on no strictly smaller closed set. 
\end{cor}

For the next corollary we denote by 
$\bfz_j  = (z
\mapsto z_j): \C^d \to \C$, $j = 1, \dots, d$, the coordinate
projections, and consider $\R^d \subseteq \C^d$ canonically.

\begin{cor}\label{spc.c.sa-mult}
Let $(\Phi,H)$ be a Borel functional calculus on $\C^d$ 
such that the operator $A_j := \Phi(\bfz_j)$ is
self-adjoint for each $j=1, \dots, d$.  Then 
$\Phi$ is concentrated on $\R^d$.
\end{cor}

\begin{proof}
By Theorem \ref{spc.t.main}, the set $\set{\bfz_j \notin \R}$ is $\Phi$-null
for each $j =1, \dots, d$. Hence, the set
$\C^d \ohne \R^d = \set{(\bfz_1, \dots, \bfz_d) \notin \R^d} =
\bigcup_{j=1}^d \set{\bfz_j \notin \R}$ is $\Phi$-null as well. 
\end{proof}

\medskip

\subsection{Multiplication Operators Revisited}\label{spc.s.mul-rev}

Let $\prO = (\Omega, \Sigma,\mu)$ be a measure space with 
associated multiplication calculus
\[ \Phi(a) := M_a \qquad (a\in \Meas(\Omega, \Sigma))
\]
as in Section \ref{s.mfc}. The measure
space $\prO$ is called {\emdf semi-finite} if
each set of infinite measure has a subset of
finite but non-zero measure.

\begin{lem}\label{spc.l.semi-finite}
Let $\prO$ be a  semi-finite measure space. 
Then for a set $A\in \Sigma$ the following assertions are equivalent:
\begin{aufzii}
\item $A$ is $\mu$-null.
\item $A$ is $\Phi$-null.
\end{aufzii}
\end{lem}

\begin{proof}
Clearly (i) implies (ii), even if $\prO$ is not semi-finite. For the
converse
suppose that $\Phi(\car_A) = 0$. Then $\car_A f = 0$ for each  $f\in
\Ell{2}(\prO)$. Then $\car_B = 0$ for all $B \subseteq A$ with $\mu(B)
< \infty$. By semi-finiteness, this implies that $\mu(A) = 0$.\end{proof}

The following is a standard result from elementary operator
theory. Here we obtain it as a corollary of functional
calculus theory. 

\begin{cor}\label{spc.c.mult}
Let $\prO = (\Omega, \Sigma,\mu)$ be a semi-finite measure space.
Then the following assertions hold for each $a\in \Meas(\Omega, \Sigma)$.
\begin{aufzi}
\item $M_a$ is injective if and only if $\mu\set{a=0} =0$.  
In this case, $M_a^{-1} = M_{a^{-1}}$. 

\item $M_a$ is bounded if and only if $a\in \Ell{\infty}(\prO)$. In 
this case $\norm{M_a} = \norm{a}_{\Ell{\infty}}$.

\item $\spec(M_a) = \Aspec(M_a) = \essran(a)$, the essential range of $a$.
%where
%\[ \essran(a) = \{ z\in \C \st \forall \veps > 0 : 
%\mu\set{ a\in \Ball(z, \veps)} > 0\} = \supp( a_*\mu ).
%\] 

\item $\lambda \in \resol(M_a) \quad \Dann\quad
R(\lambda,M_a) = M_{(\lambda - a)^{-1}}$.

\item Up to equality $\mu$-almost everywhere, $a$ is uniquely
determined by $M_a$. 

\item $M_a$ is symmetric iff it is self-adjoint iff
$a \in \R$ almost everywhere. 
\end{aufzi} 
\end{cor}

We have seen that the standard spectral properties of
multiplication operators are just special cases of 
the spectral properties  of measurable functional calculi. 
On the other hand, one can derive properties of 
measurable functional calculi from properties 
of multiplication operators. This is due to the following theorem,
which is stated here for the sake of completeness, but will not be
used at any point in following sections.

\begin{thm}\label{spc.t.mult-rep}
Let $(\Phi, H)$ be a measurable calculus on the
measurable space $(X, \Sigma)$. Then there
exists a semi-finite measure space $(\Omega, \calF, \mu)$, 
a unitary operator $U: H \to \Ell{2}(\Omega,\calF, \mu)$ and
an injective $*$-homomorphism  $T: \Meas(X, \Sigma) \to 
\Meas(\Omega, \calF)$ with 
\[  \Phi(f) = UM_{Tf}U^{-1} \qquad \text{for all $f\in \Meas(X, \Sigma)$.}
\]
Moreover,  $T$ is continuous with respect to pointwise convergence of sequences.  
\end{thm}

\begin{proof}
The proof follows well-known lines, so we only sketch it.
Details can be found in many books, e.g. in \cite[VII]{ReedSimon1}, 
or \cite[Appendix D]{HaaseFC}. 

For each $x\in H$ let $\mu_x$ be the measure  on $(X, \Sigma)$
defined  by 
\[ \mu_x(A) = \dprod{\Phi(\car_A)x}{x} \qquad (A\in \Sigma).
\]
Then
\[   \dprod{\Phi(f)x}{x} = \int_X f\, \ud{\mu_x} \qquad
(f\in \BMb(X, \Sigma)).
\]
Define 
\[ Z_x := \cls{\{ \Phi(f)x \suchthat f\in \BMb(X, \Sigma)\}}.
\]
Then 
\[  \BMb(X, \Sigma) \to Z_x, \qquad f \mapsto \Phi(f)x
\]
is isometric with respect to $\norm{\cdot}_{\Ell{2}(\mu_x)}$ and hence
extends to a unitary operator
\[ V_x: \Ell{2}(X,\Sigma,\mu_x) \to Z_x.
\]
Both spaces $Z_x$ and $Z_x^\perp$  are $\Phi(\BMb(X, \Sigma))$-invariant.
Employing Zorn's lemma, one finds  a maximal set $(x_\alpha)_\alpha$ of 
unit vectors $x_\alpha$ in $H$ such that the spaces  
$Z_{x_\alpha}$ are pairwise orthogonal and satisfy
\[ H = \ell^2-\bigoplus_\alpha Z_{x_\alpha}. 
\]
For each $\alpha$ let $X_\alpha:= X \times \{\alpha\}$ be a copy of $X$, so that the
$X_\alpha$ are pairwise disjoint. Let $\Omega := \bigsqcup_\alpha
X_\alpha$ be their disjoint union. Let $\calF$ be the largest
$\sigma$-algebra on $\Omega$ such that all inclusion maps
$X_\alpha \to \Omega$ are measurable. Let $\mu$ be the 
measure on $\calF$ defined by 
\[ \mu(B) := \sum_\alpha \mu_{x_\alpha}(B \cap X_\alpha) \qquad (B \in \calF)
\]
and let  $K := \Ell{2}(\Omega,\calF,\mu)$. Define the unitary operator 
\[  U : H \to K
\]
such that $U^{-1} = V_{x_\alpha}$ on $\Ell{2}(X_\alpha, \Sigma, \mu_{x_\alpha})
\subseteq K$.  Then define the mapping $T: \Meas(X, \Sigma) \to \Meas(\Omega, \calF)$ by
\[ (Tf)(x, \alpha)  := f(x)
 \qquad (x\in X).
\]
Then $T$ has the desired properties. Define the measurable calculus
$(\Psi, K)$ by $\Psi(f) :=  M_{Tf}$. (That is, $\Psi$ is the pull-back
of the multplication operator calculus by $T$.) Then, by construction,
the identity
\[   M_{Tf} = U\Phi(f)U^{-1}
\]
holds for all $f\in \BMb(X, \Sigma)$. Consequently, by the remarks in Section
\ref{mfc.s.det}, it must hold for all $f\in
\Meas(X,\Sigma)$. (Cf. also Lemma \ref{uni.l.bdd} below.)
\end{proof}

\section{Uniqueness}\label{s.uni}

\medskip

\subsection{Uniqueness for Measurable Calculi}

In this section we shall establish several properties
that determine
a measurable functional calculus uniquely. The first one has already been 
mentioned in Section \ref{mfc.s.det}.

\begin{lem}\label{uni.l.bdd}
Let $(\Phi, H)$ and $(\Psi,H)$ be two measurable calculi on $(X, 
\Sigma)$ such that $\Phi(f) = \Psi(f)$ for all bounded functions $f$.
Then  $\Phi = \Psi$. 
\end{lem}

Lemma \ref{uni.l.bdd}  can  be easily refined.

\begin{prop}\label{uni.p.pvm}
Let $(\Phi, H)$ and $(\Psi,H)$ be two measurable calculi on $(X, 
\Sigma)$. If the corresponding projection-valued measures
coincide, i.e., if 
$\Ee_\Phi = \Ee_\Psi$, then $\Phi = \Psi$. 
\end{prop}

\begin{proof}
It follows from the hypothesis that $\Phi$ and $\Psi$ agree on 
the linear span of characteristic functions. By (MFC5), they agree
on all bounded measurable functions, hence by Lemma 
\ref{uni.l.bdd} they
must be equal.
\end{proof}

For more refined uniqueness statements we need more information about
the set of functions on which two measurable calculi agree. 
We prepare this by looking at a slightly more general 
result about intertwining operators.

\begin{thm}\label{uni.t.inter}
Let $H, K$ be Hilbert spaces, 
let $(\Phi, H)$ and $(\Psi,K)$ be two measurable calculi 
on $(X, \Sigma)$ and let $T: H \to K$ be a bounded operator. 
Then the ``intertwining set''
\[ E := E(\Phi, \Psi; T) := \{ f\in \Meas(X, \Sigma) \suchthat 
T\Phi(f) \subseteq \Psi(f)T\}
\]
has the following properties:
\begin{aufziii}
\item $E$ is a unital $*$-subalgebra of $\Meas(X, \Sigma)$.

\item If $f\in E$ and $f\neq 0$ everywhere, then $f^{-1} \in E$. 

\item $\displaystyle f\in E \quad \iff\quad
\frac{f}{1+ \abs{f}^2} ,\,\, \frac{1}{1+ \abs{f}^2} \in E$.

\item $E$ is closed under pointwise convergence of sequences.
\item If $f\in E$, then $\abs{f} \in E$. 

\item If $f,g\in E$ are real-valued, then $f\vee g,\, f\wedge g \in
  E$.
 
\item The set $\{ A\in \Sigma \suchthat \car_A \in E\}$ is
a sub-$\sigma$-algebra of $\Sigma$.

\item If $f\in E$ then $\car_{\set{f\in B}} \in E$ for
each Borel set $B \subseteq \C$.
\end{aufziii}
\end{thm}

In order to streamline the proof, we single out a lemma first.

\begin{lem}\label{uni.l.inter-aux}
In the situation of Theorem \ref{uni.t.inter},
let $f\in \Meas(X, \Sigma)$ and $A\in \Sigma$, and define
$T_A := \Psi(\car_A)T\Phi(\car_A) \in \BL(H,K)$.
Then
\[ f\in E(\Phi, \Psi; T)\,\, \dann
\,\, f\car_A \in E(\Phi, \Psi, T_A)\,\, \dann\,\,
f \in E(\Phi, \Psi, T_A). 
\]
\end{lem}

\begin{proof}
Abbreviate $P= \Phi(\car_C)$ and $Q := \Psi(\car_C)$.
Suppose that $T\Phi(f)\subseteq \Phi(f)T$. Then
\begin{align*}
 T_A\Phi(\car_A f) & = QT \Phi(\car_A) \Phi(\car_A f)
\subseteq  QT \Phi(\car_A f) = Q T \Phi(f)\Phi(\car_A)
\\ & \subseteq Q \Psi(f) T P \subseteq \Psi(f) Q TP
= \Psi(f\car_A) T_A.
\end{align*}
This proves the first implication. Next, suppose
$T_A \Phi(\car_A f) \subseteq \Psi(\car_A f) T_A$.
Then
\begin{align*}
T_A \Phi(f) = T_A \Phi(\car_A) \Phi(f)
\subseteq T_A \Phi(\car_A f) \subseteq
\Psi(\car_A f) T_A = 
\Psi(f) \Psi(\car_A) T_A = \Psi(f) T_A,
\end{align*}
which proves the second implication.
\end{proof}

Next we recall the following important result of Fuglede.

\begin{thm}[Fuglede]\label{uni.t.fuglede}
Let $A,B$ be normal operators on the complex Hilbert
spaces $H,K$, respectively, and let $T: H \to K$
be a bounded operator such that $TA\subseteq BT$. 
Then $TA^* \subseteq B^*T$. 
\end{thm}

This theorem has been established in \cite{Fuglede1950}
(for the case $H=K$ and $A= B$, 
but the proof in general situation is not much different).
In the following we shall only employ the 
theorem in  the case that $A$ and $B$ are bounded.
For this, an elegant proof (based
on the power series functional calculus) has
been given by Rosenblum and Putnam, see \cite{Rosenblum1958}
and \cite[12.16]{RudinFA}. Cf. also 
Remark \ref{uni.r.inter}.2) below.

\medskip

We are now prepared to give the

\begin{proof}[Proof of Theorem \ref{uni.t.inter}]
1) We first note that the set $E \cap \BMb(X, \Sigma)$ is
a unital $*$-subalgebra of $\BMb(X, \Sigma)$, closed under 
bp-convergence. The closedness under
conjugation follows, in the case $\K = \C$, from (the bounded operator
version of) Fuglede's theorem. In the case $\K = \R$
it is trivial.

Now let $f,g \in E$ be arbitrary. 
For $n\in \N$ define $A_n :=\set{ \abs{f} + \abs{g} \le n}$ and 
$T_n:= T_{A_n} = \Psi(\car_{A_n}) T\Phi(\car_{A_n})$. 
Then by Lemma \ref{uni.l.inter-aux}, $f\car_{A_n}, g\car_{A_n} 
\in E(\Phi, \Psi, T_n)$. By what we have just observed,
this implies that
\[  (f+g)\car_{A_n}, (fg)\car_{A_n}, \konj{f}\car_{A_n}  \in E(\Phi,\Psi;T_n).
\]
By another application of Lemma \ref{uni.l.inter-aux} we obtain
\[ f+g, fg, \konj{f} \in E(\Phi, \Psi, T_n).
\]
Let $h$ be any one of the functions $f{+}g, fg, \konj{f}$.
Then, by what we have shown so far,
\[ T_n \Phi(h) \subseteq \Psi(h) T_n.
\]
But $T_n \to T$ strongly and hence, since $\Psi(h)$ is closed,
$T\Phi(h) \subseteq \Psi(h)T$ as desired.
This concludes the proof of 1). 
  
\prfnoi
2)\ Take $f\in E$ with  $f\neq 0$ everywhere. Then
$\Phi(f^{-1}) = \Phi(f)^{-1}$ and hence for $x,y\in H$ we have
\[ \Phi(f^{-1})x= y
\iff \Phi(f)y = x \dann
\Psi(f)Ty = Tx \iff \Psi(f^{-1})Tx = Ty.
\]
This proves the claim. 

\prfnoi
3)\  follows from 1) and 2) since $\abs{f}^2 = f \konj{f}$ and 
$f = \frac{f}{1+ \abs{f}^2} \cdot \bigl( \frac{1}{1+
  \abs{f}^2}\bigr)^{-1}$.

\prfnoi 
4)\  Suppose that $f_n \in E$ and $f_n \to f$ pointwise. 
If $(f_n)_n$ is uniformly bounded, then $f\in E$ by (MFC5). 
In the general case, note that 
\[ \frac{f_n}{1 + \abs{f_n}^2} \to \frac{f}{1 + \abs{f}^2}
\quad \text{and}\quad  
\frac{1}{1 + \abs{f_n}^2} \to \frac{1}{1 + \abs{f}^2}
\]
pointwise and boundedly. The claim now follows from 3).

\prfnoi
5)\ Since, by 1),  $E\cap \BMb(X, \Sigma)$ is a norm-closed
$*$-subalgebra of $\BMb(X, \Sigma)$,
it follows by standard arguments 
(as for instance in the proof of the
Stone--Weierstra\ss{} theorem) that 
if $f\in E$ is bounded, then $\abs{f} \in E$. 
For general $f \in E$ approximate 
$f_n := f\car_{\set{\abs{f}\le n}}
\to f$ pointwise.

\prfnoi
6)\ This follows from 5).

\prfnoi
7)\  Let $\calE := \{ A\in \Sigma \suchthat \car_A \in E\}$.
It is clear that $X\in \calE$ and $\calE$ is closed under 
taking complements and disjoint countable unions
(by (MFC5)).  
Also, $\calE$ is stable under taking finite intersections
and unions, by 6).

\prfnoi
8)\ By 7) it suffices to prove the assertion for $B$ being
any ball $B = \Ball(\lambda, \veps)$. Replacing $f$ by 
$\abs{f- \lambda\car}$ (which is possible by 1) and 5))
we may suppose that $f$ is real-valued  and $B = (-\infty, \veps)$. 
Now 
\[ \car_{\set{f< \veps}} = \lim_n n (\veps - f)^+ \wedge \car
\]
pointwise, and the claim follows from 4).    
\end{proof}

From Theorem \ref{uni.t.inter} we obtain the following
information about the coincidence set of
two calculi.

\begin{cor}\label{uni.c.coin}
Let $(\Phi, H)$ and $(\Psi,H)$ be two measurable calculi on $(X, 
\Sigma)$. Then the coincidence set
\[ E := \{ f\in \Meas(X, \Sigma) \suchthat 
\Phi(f) = \Psi(f)\}
\]
has the properties {\rm 1)--8)} listed in Theorem \ref{uni.t.inter}.
\end{cor}

\begin{proof}
Apply Theorem \ref{uni.t.inter} with $H = K$ and note that
\[ \Phi(f) = \Psi(f) \iff f \in E(\Phi, \Psi; \Id) 
\cap E(\Psi, \Phi; \Id).\qedhere
\] 
\end{proof}

\begin{rems}\label{uni.r.inter}%\label{uni.r.fuglede}
\begin{aufziii}
\item 
A more diligent reasoning in the proof of Theorem \ref{uni.t.inter} 
would show that  7) is a more or less direct 
consequence solely of the axioms (PVM1)--(PVM3)
for projection-valued measures.

\item 
Fuglede's theorem (Theorem \ref{uni.t.fuglede}) is
a corollary of assertion 1) of Theorem \ref{uni.t.inter}, given the
spectral theorem.  Actually, 
Lemma \ref{uni.l.inter-aux} and its
use in the proof of Theorem \ref{uni.t.inter} are inspired 
by Fuglede's original article \cite{Fuglede1950},
see also \cite{Rosenblum1958}.

\item  On the other hand, the bounded operator version of Fuglede's
  theorem  was used in the proof of Theorem \ref{uni.t.inter}. 
However, for the special case 
$H = K$ and $T = \Id$, i.e., for Corollary \ref{uni.c.coin}, Fuglede's
theorem is not needed. Indeed, 
the implication
\[ \Phi(f) = \Psi(f) \quad \dann\quad
\Phi(\konj{f})= \Psi(\konj{f})
\]
follows directly from (MFC4). 
\end{aufziii}
\end{rems}

\medskip

\subsection{Uniqueness for Borel Calculi}

Now  we confine ourselves to Borel functional
calculi, more precisely  to calculi on subsets $X$ of $\K^d$ endowed
with the trace of the Borel algebra. We denote by $\bfz_j  = (z
\mapsto z_j): X \to \K$, $j = 1, \dots, d$, the coordinate
projections,
and $\bfz := (\bfz_1, \dots, \bfz_d): X \to \K^d$ the inclusion
mapping.

\begin{lem}\label{uni.l.Cd}
%Let $X\subseteq \C^d$ and $\Sigma = X \cap \Borel(\C^d)$ the trace
%of the Borel algebra on $X$.  
Let  $E$ be a subset of $\Meas(\K^d)$ that satisfies 
the properties {\rm 1)--8)} listed in Theorem 
\ref{uni.t.inter}. Then the following assertions are equivalent.
\begin{aufzii}
\item $E= \Meas(X, \Sigma)$;
\item $\bfz_1, \dots,  \bfz_d \in E$;
\item $\displaystyle  
\frac{\bfz_j}{1 + \abs{\bfz}^2} \in E \quad(j=1, \dots, d)\quad
\text{and}
\quad \frac{1}{1 + \abs{\bfz}^2} \in E$;
\item $\displaystyle  
\frac{\bfz_j}{(1 + \abs{\bfz}^2)^{\frac{1}{2}}} \in E\quad(j=1, \dots, d)$.
\end{aufzii}
\end{lem}

\begin{proof}
(ii)$\dann$(i): The coordinate projections generate the Borel algebra
on  $\K^d$. By $E$ having properties 7) and 8), 
it follows that
$\car_A \in E$ for all $A\in \Sigma$. By properties 1) and 4), 
$E= \Meas(X, \Sigma)$ as desired.

\prfnoi
(iii)$\dann$(ii): This follows from properties 1) and 2) 
and the representation
\[ \bfz_j = \frac{\bfz_j}{1 + \abs{\bfz}^2} \cdot 
\Bigl(\frac{1}{1 + \abs{\bfz}^2} \Bigr)^{-1}. 
\]
(iv)$\dann$(ii): By property 1) we obtain first
\[ \frac{\abs{\bfz_j}^2}{1 + \abs{\bfz}^2} \in E \qquad (j=1, \dots,
d).
\]
From this, one concludes $\frac{1}{1 + \abs{\bfz}^2} \in E$ and then
proceeds as in the proof of the implication (ii)$\dann$(i).
\end{proof}

\begin{thm}\label{uni.t.Cd}
Let $X\subseteq \K^d$, endowed with the trace $\sigma$-algebra.
Let $(\Phi, H)$ and $(\Psi,H)$ be two measurable calculi on $X$.
Then each of the following conditions implies that $\Phi = \Psi$.
\begin{aufziii}
\item $\Phi$ and $\Psi$ agree on the functions $\bfz_1, \dots,
  \bfz_d$;
\item $\Phi$ and $\Psi$ agree on the functions 
\[  
\frac{\bfz_j}{1 + \abs{\bfz}^2}\quad(j=1, \dots, d)\quad
\text{and}
\quad \frac{1}{1 + \abs{\bfz}^2};
\]
\item $\Phi$ and $\Psi$ agree on the functions
\[  
\frac{\bfz_j}{(1 + \abs{\bfz}^2)^{\frac{1}{2}}}\quad(j=1, \dots, d);
\]
\end{aufziii}
\end{thm}

Let $A$ be a normal (self-adjoint if $\K= \R$) operator on a Hilbert space $H$ and let
$K \subseteq \K$ be a Borel subset of $\C$. A Borel calculus
$(\Phi,H)$ on $K$ is called a Borel calculus {\emdf for (the operator)
  $A$},
if $\Phi(\bfz) = A$. By Theorem \ref{uni.t.Cd} applied with $d=1$, 
a Borel calculus for $A$ is uniquely determined. We can even say a
little more.

\begin{cor}\label{uni.c.C}
Let $K, L$ be Borel subsets of $\K$ and 
let $(\Phi,H)$ and $(\Psi,H)$ be Borel functional
calculi on $K$ and $L$, respectively, for the
same operator $A$ on $H$. Then $\Phi$
and $\Psi$ are both concentrated on $K \cap L$ and 
\[ \Phi_{K\cap L} = \Psi_{K \cap L}.
\]
\end{cor}

\begin{proof}
By Theorem \ref{uni.t.Cd} applied with $d=1$, one has $\Phi_\K = \Psi_\K$. 
Hence 
\[ \Phi(\car_{K \ohne L}) = 
\Phi_\C(\car_{K \ohne L}) = 
\Psi_\C(\car_{K \ohne L}) = \Psi(0) = 0.
\]
The rest is simple.
\end{proof}

Theorem \ref{uni.t.Cd} has another consequence, already mentioned in 
Section \ref{mfc.s.pull}.

\begin{prop}\label{uni.p.uni-mult}
Let $(\Phi, H)$ and $(\Psi, H)$  be measurable functional
calculi on the measurable spaces $(X, \Sigma_X)$ and $(Y, \Sigma_Y)$,
respectively. Let
$f_1, \dots, f_d \in \Meas(X, \Sigma_X)$ and
$g_1, \dots, g_d \in\Meas(Y,\Sigma_Y)$ such that 
\[  \Phi(f_j) = \Psi(g_j)\qquad (j=1, \dots, d).
\]
Then for each $h\in \Meas(\K^d)$ one has
\[ \Phi(h \nach (f_1, \dots, f_d)) = 
\Psi(h \nach (g_1, \dots, g_d)).
\]
\end{prop}

\section{Construction of Measurable Calculi}\label{s.con}

In this section we describe different steps that lead to the
construction of a measurable functional calculus. In the
results we have in mind one starts
with a ``partial calculus'' so to speak.  That is, one is given a
subset $M \subseteq \Meas(X, \Sigma)$, in the following 
called our {\emdf set of departure}, and a mapping
$\Phi : M \to \Clo(H)$ that has the properties of a
restriction of a measurable calculus.
And one aims at asserting that this partial calculus 
is in fact such a restriction, that is, 
can be extended (uniquely, if possible) to a full measurable
calculus. 

In a sense, the spectral theorem itself is of this
form. There, $X = \K$, the set of departure
is $M = \{\bfz\}$  the coordinate
mapping, and the only requirement is that the operator
$\Phi(\bfz)$ is normal (self-adjoint if $\K= \R$).  

\medskip
\noindent
In all what follows, 
$(X, \Sigma)$ is a measurable space and $H$ a Hilbert space.

\medskip

\subsection{From Bounded to Unbounded Functions (Algebraic Extension)}\label{con.s.algext}

Here we take $M := \BMb(S, \Sigma)$, the bounded
measurable functions, as our set of departure.
We know already that each measurable functional
calculus on $(X, \Sigma)$ is uniquely determined 
by its restriction to $M$. 

But more is true: each functional calculus  defined 
originally on $\BMb(X, \Sigma)$ can be uniquely extended to a full
measurable functional calculus. The procedure for this is canonical
and  known as  ``algebraic extension'' or ``extension  by (multiplicative)
regularization''.

\begin{thm}\label{con.t.ext-bdd}
Let $(X, \Sigma)$ be a measurable space, 
$H$ a Hilbert space and 
\[ \Phi: \BMb(X, \Sigma) \to
\BL(H)
\]
a unital and (weakly) bp-continuous $*$-homomorphism. 
Then $\Phi$ extends uniquely to a measurable functional calculus
$\Meas(X, \Sigma) \to \Clo(H)$.
\end{thm}

\begin{proof}
Uniqueness is clear. For existence, let $\Phi: \BMb(X,\Sigma) \to
\BL(H)$ be as stated in the theorem. If $f\in \Meas(X, \Sigma)$
is arbitrary, we take any {\em anchor element}%
\footnote{Such elements were called ``regularizers'' in  
\cite{HaaseFC}, but the terminology has been modified 
in the meantime, cf. \cite{Haase2020pre}.}
for $f$ in $\BMb(X, \Sigma)$, i.e., a
function $e\in \BMb(X, \Sigma)$ such that $ef$ is
bounded and $\Phi(e)$ is injective. (The function $e=
(1+\abs{f})^{-1}$ will do, see below.)
Then define
\[  \widetilde{\Phi}(f) := \Phi(e)^{-1} \Phi(ef).
\]
It is easy to see that this definition does not depend on the choice 
of the  function $e$ and the so-defined mapping 
$\widetilde{\Phi} : \Meas(X, \Sigma) \to \Clo(H)$
extends $\Phi$ and satisfies (MFC1)--(MFC3).  As a matter of fact,
it also satisfies  (MFC4) and (MFC5), hence it is a
measurable functional calculus.

\prfnoi
It remains to show 
that an anchor element as above can always be found. 
To this end, for given $f\in \Meas(X, \Sigma)$  define
\[   e := \frac{1}{1 + \abs{f}} \quad \text{and}\quad e_n :=
\car_{\set{\abs{f}\le n}} \quad (n \in \N).
\]
Clearly,  $e$ and $ef$ are both bounded functions.
Also,  $e_n \abs{f}$ is bounded and
\[ e_n = \bigl( (1 +\abs{f})e_n \bigr) e,
\]
which leads to $\Phi(e_n)= \Phi( (1 +\abs{f})e_n) \Phi(e)$.
Since  $\Phi(e_n) \to \Phi(\car) = \Id$ strongly (by (MFC1) and
(MFC5')),  $\Phi(e)$ must be  injective.
\end{proof}

Theorem \ref{con.t.ext-bdd} 
tells that in order to establish a measurable functional
calculus it suffices to construct its restriction to bounded
functions. The rest is canonical.

\medskip

The following example
shows that without the assumption of 
bp-continuity in Theorem \ref{con.t.ext-bdd}
one can encounter quite degenerate situations. 
(I am indebted to Hendrik Vogt for providing the
main idea.)

\begin{exa}\label{con.exa.hendrik}
Let $\K=\ \C$, $X = \N$ and $\Sigma = \Pot{\N}$, the whole power set. Then 
\[ \BMb(X,
\Sigma) = \ell^\infty\quad  and \quad
\Meas(X, \Sigma) = \C^\N, 
\]
the space of all sequences.

For each strictly increasing mapping 
(``subsequence'')  $\pi: \N \to \N$ pick
a non-zero multiplicative functional
$\Phi_\pi : \ell^\infty \to \C$ which vanishes 
on the ideal of sequences $x = (x_n)_n \in \ell^\infty$ 
such that $\lim_{n \to \infty} x_{\pi(n)} = 0$. 
This exists: by Zorn's lemma there is a maximal ideal $M_\pi$
containing this ideal 
and by the Gelfand--Mazur theorem $\ell^\infty/M_\pi 
\cong \C$ as 
Banach algebras. By the commutative Gelfand--Naimark theorem,
$\Phi_\pi$ is a unital $*$-homomorphism. (Alternatively one can define
$\Phi_\pi$ as the ultrafilter limit with respect to some 
ultrafilter that contains all the ``tails''
 $\{\pi(k) \suchthat k \ge n\}$ for $n \in \N$.) 

Now let $I$ be the set of all such subsequences $\pi$, 
let $H := \ell^{2}(I)$ and define 
\[ 
\Phi: \ell^\infty \to  \ell^\infty(I) \subseteq \BL(H),\qquad 
\Phi(x) := (\Phi_\pi(x))_\pi,
\]
where we identify a bounded function on $I$ with the
associated multiplication operator on $H= \ell^2(I)$. Then 
$\Phi$ is a unital $*$-homomorphism. 

If $f : \N \to \C$ is any unbounded sequence, then
there is a subsequence $\pi\in I$ along which 
$\abs{f}$  converges  to $+\infty$. Hence, if 
$e \in \ell^\infty$ is such that $ef \in\ell^\infty$ as well,
then $e(n)$ converges to zero along $\pi$. Consequently,
$\Phi_\pi(e) = 0$. Let $\delta_\pi \in H$ be the
canonical unit vector which is $1$ at $\pi$ and $0$ else.
Then $\Phi(e) \delta_\pi = \Phi_\pi(e)\delta_\pi = 0$. This 
not only shows that $f$ does not admit any ``anchor elements'',
but even more: the set
\[ [f]_{\ell^\infty} := \{ e\in \ell^\infty \suchthat ef \in \ell^\infty\}
\]
is not an ``anchor set'' (in the terminology of \cite{Haase2020pre}). 
It follows that algebraic extension, even in its
more general form discussed in \cite{Haase2020pre}, does not
lead to a proper extension of the original calculus. 
Of course, $\Phi$ is not bp-continuous.
\end{exa}

\medskip

\subsection{From Projection-Valued Measures to Measurable Functional Calculus}\label{con.s.ext-pvm}

Next, we take $M = \{ \car_A \suchthat 
A\in \Sigma\}$, the set of all characteristic functions, 
as our set of departure. 
In other words, we start with a projection-valued measure. 

\begin{thm}\label{con.t.ext-pvm}
Let $(X,\Sigma)$ be a measurable space, let
$H$ be a Hilbert space, and let
$\Ee: \Sigma \to \BL(H)$ be a projection-valued
measure as  defined in Section \ref{s.pvm}. 
Then there exists a unique measurable functional
calculus $\Phi : \Meas(X, \Sigma) \to \Clo(H)$
such that  $\Ee(A) = \Phi(\car_A)$ for 
each $A\in \Sigma$. 
\end{thm}

\begin{proof}
It follows from the axioms that if $A, B \in \Sigma$ are disjoint then
the ranges of 
$\Ee(A)$ and $\Ee(B)$ are orthogonal. 
Define $\Phi$ on simple functions $f$ by 
\[ f = \sum_{j=1}^n \alpha_j \car_{A_j}
\quad\dann\quad \Phi(f) := \sum_{j=1}^n \alpha_j \Ee(A_j)
\]
where $A_1,\dots, A_n$ is any finite measurable
partition of $X$ and $\alpha_1, \dots, \alpha_n$ are
scalars. By standard arguments it is shown that
$\Phi$ is a (well-defined) contractive unital $*$-homomorphism. Hence
$\Phi$ extends continuously to $\BMb(X, \Sigma)$, 
and this extension, again denoted by $\Phi$,
is still a contractive  unital $*$-homomorphism. 

In view of Theorem \ref{con.t.ext-bdd} it suffices to show that
$\Phi$ is weakly  bp-continuous. For any pair
$x,y\in H$ the mapping
\[ \mu_{x,y}: \Sigma \to \K,\qquad 
\mu_{x,y}(A) := \dprod{\Ee(A)x}{y}
\]
is a $\K$-valued measure. Obviously,
\[ \dprod{\Phi(f)x}{y} = \int_X f \,\ud{\mu_{x,y}}
\]
for each $f\in \BMb(X, \Sigma)$. Therefore,
(MFC5) is a consequence of the dominated convergence
theorem, and the proof is complete. 
\end{proof}

\medskip

\subsection{From Continuous to Measurable Functional Calculus}%
\label{con.s.ext-cont}

In this section we confine ourselves  
to a compact Hausdorff space $X$ endowed
with the Borel $\sigma$-algebra  $\Sigma = \Borel(X)$.
The set of departure is $M = \Ce(X)$, the space of
continuous functions.% vanishing at infinity. 

\begin{thm}\label{con.t.ext-cont}
Let $X$ be a compact Hausdorff space
and let $\Phi: \Ce(X) \to \BL(H)$ be a unital 
$*$-homomorphism. Then $\Phi$ extends
uniquely to a measurable calculus $\Phi$ on 
$(X, \Borel(X))$ with the additional property:
\[  \dprod{\Phi(f)x}{x} = \sup\{ \dprod{\Phi(g)x}{x} \suchthat 
g\in \Cc(X), 0 \le g \le f\} \qquad (x\in H)
\]
whenever $f\ge 0$ is a bounded and lower semi-continuous function on $X$. 
%
%%If $X$ is second countable, the additional property is automatic, and
%the extension of $\Phi$ to a measurable calculus on 
%$X$ is unique without any additional requirement. 
\end{thm}

\begin{proof}
Existence:\ This is rather standard, so we just give a sketch. For more
details cf. the proof of Theorem 12.2 in \cite{RudinFA}. 

\prfnoi
As in the proof of Theorem \ref{mfc.t.prop1} one shows that 
$\Phi$ is contractive. It follows that for all
$x,y\in H$ the linear functional
$f\mapsto \dprod{\Phi(f)x}{y}$
is bounded. By the Riesz--Markov--Kakutani
representation theorem, there is a unique
regular $\K$-valued Borel measure $\mu_{x,y} \in \eM(X)$ such 
that 
\[  \dprod{\Phi(f)x}{y} = \int_X f\, \ud{\mu_{x,y}} 
\qquad (f\in \Ce(X), \, x,y,\in H).
\]
One easily shows that the mapping $(x,y)\mapsto \mu_{x,y}$
is sesquilinear (bilinear if $\K= \R$).  Given $g \in \BMb(X, \Borel(X))$, 
the sesquilinear/bilinear form
\[ (x, y ) \mapsto \int_X g \, \ud{\mu_{x,y}}
\]
is bounded. By a standard result from Hilbert space theory,
there is a unique operator $\Psi(g)$ such that 
\[ \dprod{\Psi(g)x}{y} = \int_X f\, \ud{\mu_{x,y}}
\]
for all $x,y\in H$. It is then routine to show that 
$\Psi: \BMb(X, \Borel(X)) \to \BL(H)$
is a weakly bp-continuous unital $*$-homomorphism that
extends $\Phi$. Moreover, it follows 
from the regularity of the measures $\mu_{x,x}$ that 
$\Psi$ has the additional property asserted
in the theorem.  By Theorem \ref{con.t.ext-bdd}, 
$\Psi$ extends to a full measurable
calculus.

\prfnoi
Uniqueness:\ Let $\Psi_1, \Psi_2$ be two extensions of $\Phi$ that
both have the additional property.
By Theorem \ref{uni.t.inter}, the set $\{ A\in \Sigma \suchthat
\Psi_1(\car_A) = \Psi_2(\car_A)\}$ is a $\sigma$-algebra.
By the additional property, this $\sigma$-algebra contains
each open set, and hence coincides with $\Borel(X)$. 
It follows that $\Psi_1= \Psi_2$. 
\end{proof}

If the compact space $X$ is metrizable, 
each open subset is $\sigma$-compact, and each bounded and  
positive lower semicontinuous function is the 
pointwise limit of a uniformly bounded  sequence of continuous
functions. It follows that in this case the additional
property is automatic, and the uniqueness assertion
holds without that requirement.

\begin{cor}\label{con.c.cont-metr}
Let $X$ be a compact and  metrizable space
and let $\Phi: \Ce(X) \to \BL(H)$ be a unital 
$*$-homomorphism. Then $\Phi$ extends
uniquely to a measurable calculus $\Phi$ on 
$(X, \Borel(X))$. 
\end{cor}

\begin{rem}\label{con.r.cont-baire}
The {\emdf Baire algebra} $\Baire(X)$ on a compact Hausdorff space $X$ 
is the smallest $\sigma$-algebra that renders each continuous
function measurable. It coincides with the Borel algebra
when $X$ is metrizable, but is generally different from it. 

A measure defined on the Baire algebra is called a {\emdf Baire
  measure}. Baire measures are automatically regular and 
uniquely determined by their associated
linear functionals on $\Ce(X)$. By using Baire measures instead
of regular Borel measures, one sees that Corollary
\ref{con.c.cont-metr}
stays true if one drops metrizability but replaces $\Borel(X)$ by 
$\Baire(X)$. 
\end{rem}

\medskip
\subsection{Cartesian Products}\label{con.s.prod}

In this last section  we look at measurable functional calculi 
on  Cartesian products, that is, tensor products 
of functional calculi. At least in a special case, we have 
a positive result.

\begin{thm}\label{con.t.prod}
Let  $(\Phi, H)$ and $(\Psi, H)$ be
Borel functional calculi on the compact metric spaces
$X$ and $Y$, respectively. Suppose that
these calculi commute, in the sense that
\[ \Phi(f)\Psi(g) = \Psi(g)\Phi(f)\qquad
(f \in \Ce(X),\,\, g\in \Ce(Y)).
\]
Then there is a unique Borel calculus
$(\Phi \tensor \Psi, H)$ on 
$X\times Y$
such that 
\[ (\Phi\tensor \Psi)(f\tensor g) = 
\Phi(f) \Psi(g)
\]
for all $f\in\BMb(X)$ 
and $g \in \BMb(Y)$.
\end{thm}

\begin{proof}
Uniqueness follows from Corollary \ref{con.c.cont-metr}. 
For existence, observe first 
that by Theorem \ref{uni.t.inter} and the hypothesis on has
\beq\label{con.eq.prod-comm}
\Phi(f) \Psi(g) = \Psi(g) \Phi(f)
\eeq
for all $f\in \BMb(X)$ and $g\in \BMb(Y)$. Now let
\[ \calE := \spann\{ \car_{A \times B} \suchthat
A\in \Borel(X),\, B \in \Borel(Y)\}
\]
and  define a linear mapping 
\[ \Lambda: \calE 
\to \BL(H),\qquad \Lambda(\car_{A\times B}) := \Phi(\car_A)\Psi(\car_B).
\]
(Of course, one has to show that this map is well defined.)
From \eqref{con.eq.prod-comm}  it follows that
$\Lambda$ is a unital $*$-homomorphism, and
since $\calE$ is closed under taking square roots of
positive functions, $\Lambda$ is contractive
(cf.{ }the proof of Theorem \ref{mfc.t.prop1}.e). Hence,
$\Lambda$ extends uniquely to a bounded operator
(again denoted by $\Lambda$ 
on the $\norm{\cdot}_\infty$-closure $\cls{\calE}$ of 
$\calE$. As a matter of fact, this extension is still
a unital $*$-homomorphism.  

By the Stone--Weierstra\ss{} theorem,
$\Ce(X\times Y)$ is the closure of $\Ce(X) \tensor \Ce(Y)$
and hence
contained in $\cls{\calE}$. So we 
may apply Corollary \ref{con.c.cont-metr} 
to obtain an extension
of $\Lambda$, denoted by $\Phi \tensor \Psi$, to a Borel
functional calculus on $X\times Y$. 

The mapping $f\mapsto (\Phi \tensor \Psi)(f \tensor \car)$
is a Borel calculus on $X$ (it is the
pull-back with respect to the mapping $f \mapsto f \tensor \car$)
and coincides with $\Phi$
on $\Ce(X)$. It follows that 
\[ (\Phi \tensor \Psi)(f \tensor \car) = \Phi(f) 
\]
for all $f\in \BMb(X)$. Analogously, one obtains
$(\Phi \tensor \Psi)(\car \tensor g) = \Psi(g)$
for all $g\in \BMb(Y)$. This implies
\[ (\Phi\tensor\Psi)(f\tensor g) =
(\Phi\tensor\Psi)\Bigl((f\tensor \car) (\car \tensor g)\Bigr)
= \Phi(f) \Psi(g),
\]
as desired. 
\end{proof}

\begin{rem}\label{con.r.prod-baire}
Recall from Remark \ref{con.r.cont-baire} 
that one can allow for
non-metrizable spaces in Corollary \ref{con.c.cont-metr} 
when
one uses the Baire instead of the Borel algebra. 
A similar remark applies to Theorem  \ref{con.t.prod}. 
Continuing in  this line of thought, by adapting the proof
of Theorem \ref{con.t.prod} one obtains the following
generalization to arbitrary products:

\smallskip
\noindent
{\bf Theorem:} {\em 
Let, for each $\lambda \in \Lambda$,  $(\Phi_\lambda, H)$ be a
Baire functional calculus on the compact Hausdorff space
$X_\lambda$. Suppose that all 
these calculi commute, in the sense that
\[ \Phi_\lambda(f)\Phi_\mu(g) = \Phi_\mu(g)\Phi_\lambda(f)
\]
for all $\lambda,\: \nu \in \Lambda$,  
$f \in \BMb(X_\lambda,\Baire(X_\lambda))$ and $g\in \BMb(X_\mu, \Baire(X_\mu))$.

Then there is a unique Baire calculus
$(\Psi, H)$ on 
$\prod_{\lambda \in \Lambda} X_\lambda$
such that 
\[ \Psi( \otimes_\lambda f_\lambda)
=  \prod_{\lambda} \Phi_\lambda(f_\lambda)
\]
for all $f_\lambda\in\BMb(X_\lambda)$  with 
$f_\lambda = \car$ for all but finitely many $\lambda\in \Lambda$.}

\smallskip
\noindent
As a matter of fact, there is an analogue for arbitrary
products of  Borel calculi
when one makes appropriate assumptions about positive
lower semi-continuous functions as in Theorem 
\ref{con.t.ext-cont}.
\end{rem}

\section{The Spectral Theorem}\label{s.spt}

Finally, we shall state and prove ``our'' version(s)
of the spectral theorem.

\medskip
\subsection{Bounded Operators, Complex Case}

We start with the bounded
operator version in the case $\K= \C$.

\begin{thm}[Spectral Theorem: Bounded Operators, $\K=\C$]\label{spt.t.bdd}
Let $A_1, \dots, A_d$ be bounded normal and pairwise
commuting  operators on a complex Hilbert space $H$.
Then there is a unique Borel calculus $(\Phi,H)$ on $\C^d$ such that
$\Phi(\bfz_j) = A_j$ for all $j =1, \dots, d$. 
\end{thm}

\begin{proof}
Uniqueness follows from Theorem \ref{uni.t.Cd}, 
so we prove existence.  By Fuglede's theorem, the operators $A_1, \dots, A_d$ generate a
commutative unital $C^*$-subalgebra $\calA$ of $\BL(H)$.
By Gelfand's theorem, there is a compact space $X$ and
an  isometric isomorphism $\Psi: \Ce(X) \to \calA$ of
$C^*$-algebras. By Theorem \ref{con.t.ext-cont} this map extends
to a Borel calculus $(\Psi, H)$ on $X$. 
Let $f_j \in \Ce(X)$ be such that $\Phi(f_j) = A_j$
for $j =1 , \dots, d$. Let $(\Phi, H)$ be the 
push-forward of $\Phi$ along the continuous mapping
\[ f = (f_1, \dots, f_d) : X \to \C^d.
\]
Then $(\Phi, H)$ is a measurable calculus such that 
\[ \Phi(\bfz_j) = \Psi( \bfz_j \nach f) = \Psi(f_j)
= A_j \quad \text{for each $j =1, \dots, d,$}
\]
as desired.
\end{proof}

The abovegiven proof rests on Gelfand's theorem. 
If one wants to avoid this result, one may proceed as 
follows.  In a first step, the theorem is reduced to the case of 
self-adjoint operators.
Each normal operator $A_j$ can be written uniquely as
\[ A_j = A_{j 1} + \ui A_{j 2},
\]
where the operators  $A_{j1}$ and $A_{j2}$ are self-adjoint.  Also, the
operators $A_{j k}$  ($k=1,2$, $j = 1, \dots, d$) are pairwise commuting.
Suppose that Theorem \ref{spt.t.bdd} is known provided all operators are
self-adjoint. Then we obtain a Borel functional calculus $\Psi$ on $\C^{2d}$
such that  $\Psi(\bfz_{j k}) = A_{jk}$ for all $j =1, \dots, d$ and $k
= 1,2$.  By Corollary \ref{spc.c.sa-mult}, $\Psi$ is concentrated on $\R^{2d}$
and hence can be regarded as a Borel calculus on $\R^{2d}$. Write
$\bfx_j := \bfz_{j1}$ and $\bfy_j := \bfz_{j2}$, as these coordinate functions are real-valued now. 
Identify $\R^{2d}$ with $\C^d$ via the mapping $\vphi := (x_1, y_1, \dots,
x_d, y_d) \mapsto (x_1 + \ui y_1, \dots, x_d + \ui y_d)$ and let
$\Phi$ be the push-forward of $\Psi$ along $\vphi$. Then
$\Phi$ is a Borel calculus on $\C^d$ and 
\[  \Phi(\bfz_j) = \Psi(\bfx_j + \ui \bfy_j) = \Psi(\bfx_j) + \ui
\Psi(\bfy_j)
= A_{j1} + \ui A_{j2} = A_j \qquad (j =1, \dots, d).
\]
Next, suppose that the theorem is true for $d=1$, and let
$\Phi_j$ be the Borel calculus on $\C$ (concentrated
on $\R$) such that $\Phi(\bfz) = A_j$. Since
the $A_j$ are pairwise commuting, it follows from 
Theorem \ref{uni.t.inter} that the associated functional
calculi $\Phi_j$ are pairwise commuting, too. Therefore, one
may apply Theorem \ref{con.t.prod} to obtain the ``joint functional
calculus'' $\Phi$.

This leaves us to prove Theorem  \ref{spt.t.bdd} for the
case where $d=1$ and $A_1= A$ is self-adjoint. 
In this situation there is
a remarkably elementary proof, which was already known
to Halmos \cite{Halmos1963}.
For convenience, we give the short argument.

\begin{proof}[Proof of Theorem \ref{spt.t.bdd} for 
a single self-adjoint operator]
%We assume $\K = \C$, otherwise we pass to the complexification as
%in the proof above. 
Let $A$ be a bounded, self-adjoint 
operator on $H$ and let  $a, b \in \R$ such that 
$\spec(A)\subseteq [a,b]$. 
For $p \in \C[z]$ denote by $p^*$ the polynomial
$p^*(z) := \konj{p(\konj{z})}$, and let $q := p^*p$.  
By the spectral inclusion theorem for polynomials, 
\[ \spec(q(A)) \subseteq q(\spec(A)).
\]
Now observe that $p(A)^*p(A) = p^*(A)p(A)= q(A)$. Hence,
$q(A)$ is self-adjoint and therefore its norm equals its spectral
radius (see \cite[Section 13.2]{HaaseFA} for an elementary
proof). Since $q= \abs{p}^2$ on $\R$,   
\begin{align*}
 \norm{p(A)}^2 & = \norm{p(A)^*p(A)} =
\norm{q(A)} = r(q(A)) = \sup\{ \abs{\lambda} \suchthat \lambda
\in \spec(q(A))\} 
\\ & \le
\sup\{ \abs{q(\mu)} \suchthat \mu \in \spec(A)\}
\le \norm{q}_{\infty, \spec(A)} \le \norm{p}^2_{\infty, [a,b]}.
\end{align*}
It follows that the polynomial functional calculus for $A$
is contractive for the supremum-norm on $[a,b]$. By 
Weierstrass' theorem, the polynomials are dense in $\Ce[a,b]$, and
hence there is a contractive linear map
\[ \Phi: \Ce[a,b] \to \BL(H)
\]
such that $\Phi(p) = p(A)$ for $p \in \C[z]$. It is easily
seen that $\Phi$ is a unital $*$-homomorphism. 
By Corollary \ref{con.c.cont-metr}, $\Phi$ extends uniquely to a
Borel functional calculus on $[a,b]$, and  pushing
that forward along the inclusion map, we obtain
the desired Borel functional calculus on $\C$. 
\end{proof}

\begin{rem}\label{spt.r.elem}
A simlilarly elementary proof, which even yields the better estimate
$\norm{p(A)}\le \norm{p}_{\infty, \spec(A)}$ for polynomials $p$, 
is given
in the lecture notes \cite[Theorem E.3]{ArendtVogtVoigtFMEE} by
Arendt, Vogt and Voigt.  It is inspired by the proof of Riesz and
Sz.-Nagy in \cite[VII, 106]{RieszNagyFA}.
Compare this also with Lang's approach in
\cite[XVIII, \S 4]{LangRFA}.
\end{rem}

Applied to a single operator,
Theorem \ref{spt.t.bdd} tells that 
each bounded normal operator $A$ on a
complex Hilbert space $H$ comes with 
a unique Borel calculus $(\Phi_A,H)$ 
on $\C$ such that $\Phi_A(\bfz) = A$. 
This calculus is concentrated on 
$\spec(A)$ (Corollary \ref{spc.c.single}), but also on
$\spec(A)\ohne\{\lambda\}$ whenever
$\lambda$ is not in the point spectrum
of $A$ (cf. Theorem \ref{spc.t.main}). Since
$\spec(A)$ is the smallest closed set
on which $\Phi$ is supported, we conclude
that isolated points of the spectrum must be
eigenvalues. (This can, of course, be proved in a more elementary way.)

Uniqueness of the calculus 
justifies the common habit of writing
\[ f(A) := \Phi_A(f)  \qquad (f\in \Borel(\C)).
\]
Suppose that $f(A)$ is again bounded.
Then one has the {\emdf composition rule} 
\beq\label{spt.eq.cr}  
g(f(A)) = (g\nach f)(A) \qquad (g\in \Meas(\C))
\eeq
simply because, by uniqueness, the
push-forward along $f$ of $\Phi_A$
must coincide with $\Phi_{f(A)}$.

\medskip
\subsection{Bounded Operators, Real Case}

We now consider the case that $\K= \R$. Here,
the operators we start with need to be  self-adjoint.
(The reason is  that any measurable
functional calculus maps real functions to self-adjoint operators but,
as the example
\[ A = \begin{pmatrix}  0 & 1 \\ -1 & 0
\end{pmatrix} \quad \text{on}\quad H = \R^2
\]
shows, a normal operator on a real Hilbert space need not be
self-adjoint. Hence,  in the real case normality is not
strong enough to imply the spectral theorem.)

\begin{thm}[Spectral Theorem: Bounded Operators, $\K= \R$]\label{spt.t.bdd-R}
Let $A_1, \dots, A_d$ be bounded self-adjoint and pairwise
commuting  operators on a real Hilbert space $H$.
Then there is a unique Borel calculus $(\Phi,H)$ on $\R^d$ such that
$\Phi(\bfz_j) = A_j$ for all $j =1, \dots, d$. 
\end{thm}

\begin{proof}
Complexify $H$ to $H^\C:= H \oplus \ui H$ and 
let $A_j^\C$ be the canonical $\C$-linear extension of $A_j$  to 
$H^\C$. Then the $A_j^\C$ are bounded, pairwise commuting self-adjoint
operators on $H^\C$.  
Let $(\Psi, H^\C)$ be the associated Borel
calculus on $\C^d$. By Corollary \ref{spc.c.sa-mult}, 
$\Psi$ is concentrated on $\R^d$. So, effectively,
$\Psi$ is a Borel  calculus on $\R^d$. 

Next, restrict $\Psi$ to real-valued functions, view $H^\C$ as a real
Hilbert space and let $\Phi$ be the part of $\Psi$ in the
real subspace $H \oplus \{0\} \subseteq H^\C$. We claim that $\Phi$ is
a Borel functional calculus.
To prove this, let $P$ be the orthogonal projection with range
$H$ (i.e.,  projection onto the first component). By construction and the self-adjointness of
the operators $A_j^\C$, $P\Psi(\bfz_j) = PA_j^\C = A_j^\C P =
P \Psi(\bfz_j)$ for all $j = 1, \dots, d$. By Lemma \ref{uni.l.Cd},
$P\Psi(f) = \Psi(f)P$ holds for all $f\in \Meas(X, \Sigma; \R)$. 
Hence, Lemma \ref{mfc.l.prt} tells that $\Phi$ is a measurable
functional calculus. 

Finally, observe that 
\[ \Phi(\bfz_j) = A_j^\C\cap (H \oplus\{0\})\oplus (H \oplus\{0\}) = 
 A_j
\]
for each $j= 1, \dots, d$, and we are done. 
\end{proof}

An alternative to the given proof
proceeds as follows.  Let $\calA$ be the real unital 
$C^*$-subalgebra
of $\BL(H)$, generated by the operators $A_1, \dots, A_d$. 
We can view $\calA$ as a subset of $\BL(H^\C)$ (via the isometric
embedding  $A \mapsto A^\C$ as in the proof above). 
By the following corollary of Gelfand's theorem, 
communicated to us by Jürgen
Voigt, there is a compact Hausdorff space $K$ and an isometric isomorphism
$\Psi: \Ce(K;\R) \to \calA$.  Now proceed exactly as 
in the proof of Theorem \ref{spt.t.bdd}.

\begin{prop}
Let $\calB$ be  unital $C^*$-algebra and $\calA \subseteq \calB$
a real, closed, unital and commutative $*$-subalgebra of $\calB$
consisting entirely of self-adjoint elements.
Then  there is a compact Hausdorff space $K$ and a unital and 
isometric $*$-isomorphism $\Psi: \Ce(K;\R) \to \calA$. 
\end{prop}

\begin{proof}
Let $\calA^\wedge := \calA + \ui \calA$. Then $\calA^\wedge$ is
a unital, commutative, $*$-subalgebra of $\calB$. Moreover,
it is closed, since $\calA$ is closed and the 
mapping 
\[ c = a + \ui b \mapsto (a,b) = (\tfrac{1}{2}(c+c^*), \tfrac{1}{2\ui}(c-c^*) )
\]
is bounded. Then, by Gelfand's theorem, there is a compact Hausdorff
space $K$ and a unital and isometric $C^*$-isomorphism
$\Psi: \Ce(K;\C) \to \calA^\wedge$. Clearly,
$\Psi(\Ce(K;\R)) = \calA$. 
\end{proof}

Actually, in order to 
arrive at a continuous calculus
$\Psi: \Ce(K;\R) \to \calA$ passing to a complexification
is not necessary. Instead, one may apply one of the
existing purely real characterizations of real $\Ce(K)$-spaces,
see e.g., \cite{AlbiacKalton2007}.

\medskip

Finally, there is an alternative route to Theorem \ref{spt.t.bdd-R}
avoiding  both complexification and Gelfand-type theorems. 
As in the complex case, one can reduce the theorem
to the case $d=1$ and the boundedness of the real polynomial calculus. 
The latter can be obtained, e.g., by the proofs given in 
\cite[Theorem E.3]{ArendtVogtVoigtFMEE} or
\cite[XVIII, \S 4]{LangRFA}, already mentioned in Remark \ref{spt.r.elem} above.

\medskip

\subsection{Unbounded Operators}

The spectral theorem for (in general) unbounded operators
shall be reduced to the one for bounded operators.
To this aim,  we introduce for
any densely defined closed operator $A$ on a $\K$-Hilbert space $H$ 
the bounded operators
\beq 
T_A := (1 + A^*A)^{-1}\quad\text{and}\quad
 S_A := AT_A = 
A (1+ A^*A)^{-1}.
\eeq
Note that {\em if} $\Phi$ is a Borel calculus on $\K$ for 
$A$, then 
\[ T_A= \Phi\Bigl( (1 + \abs{\bfz}^2)^{-1}\Bigr) 
\quad \text{and}\quad
S_A = \Phi\Bigl( \bfz (1 + \abs{\bfz}^2)^{-1}\Bigr).
\]
The idea is, roughly, to apply Theorem \ref{spt.t.bdd}
to the operators $T_A$ and $S_A$ and then
construct a Borel calculus for $A$ as a push-forward.
In order to succeed with this idea, we need 
the following properties of the operators $T_A$ and $S_A$.

\begin{lem}\label{spt.l.TASA}
Let $A$ be  a densely defined and closed operator on a Hilbert space
$H$. Then the operators $T_A,S_A$ have the following properties:
\begin{aufzi}
\item $T_A$ is an injective, bounded and positive 
self-adjoint operator with $\norm{T_A}\le 1$;
$S_A$ is a bounded operator.

\item $A = \cls{S_AT_A^{-1}}$.

\item If $A$ is normal then $T_{A^*} = T_A$ and 
$S_{A^*}= S_A^*$, and one has $A = T_A^{-1}S_A$. Moreover,
$T_A S_A = S_A T_A$ in this case.

\item If $A$ is self-adjoint or normal, then so is  $S_A$.

\end{aufzi} 
\end{lem} 

\begin{proof}
a)\  This is standard Hilbert space operator theory, see
\cite[Theorem 13.13]{RudinFA}.
%Since
%$\dom(A^*A) \subseteq \dom(A)$, the operator $S_A=AT_A$ i% bounded. 

\noindent
b)\  $S_A T_A^{-1} = A T_A T_A^{-1} = A \res{D}$, where
$D = \ran(T_A) = \dom{A^*A}$ is a core for $A$
\cite[Theorem 13.13]{RudinFA}.

\prfnoi
c)\ Suppose that $A$ is normal. Then $(A^*)^* = A$ 
 since $A$ is closed, and hence
 $T_{A^*}= (1+ (A^*)^*A^*)^{-1} = (1 + A A^*)^{-1}
= (1 + A^*A)^{-1} = T_A$. Next, we claim that
\beq\label{spt.eq.TASA-aux}   
T_A A \subseteq AT_A = S_A.
\eeq
Proof of claim:  Let 
$x\in \dom(A)$ and $y := T_Ax$. Then $y + A^*Ay = x$ and
hence $A^*Ay = x - y \in \dom(A)$. Applying $A$ 
and using the normality we obtain
\[ Ax = Ay + AA^*Ay = (\Id + AA^*)Ay = (\Id + A^*A) Ay,
\]
which results in  $T_AAx = Ay = AT_Ax = S_Ax$.

A consequence of  \eqref{spt.eq.TASA-aux} is that
\[ T_A S_A = T_A A T_A = A T_A T_A = S_A T_A
\]
since $\ran(T_A) \subseteq \dom(A)$. Next, as
$\dom(A)$ is dense, 
\[ S_A^* = (T_A A)^* = A^* T_A = A^* T_{A^*} = S_{A^*}.
\]
Furhermore, we obtain $A \subseteq T_A^{-1}S_A$ from \eqref{spt.eq.TASA-aux}. In order
to establish equality here, let $x \in \dom(T_A^{-1}S)$, i.e, $S_A x =
AT_Ax \in \ran(T_A) = \dom(A^*A) = \dom(AA^*)$. Then
\[ x = (\Id + A^*A)T_Ax = T_A x + A^*(A T_A x) 
\in \dom(A),
\]
as desired. 

\prfnoi
d)\ If $A$ is self-adjoint, then $S_A^* = S_{A^*}
= S_A$ by c).  If $A$ is normal then
\begin{align*}
 S_A^* S_A & = S_{A^*}S_A = A^*T_{A^*} A T_A
= A^* T_A A T_A  
= A^* A T_A^2 = A A^* T_{A^*}^2 
\\ & = A T_{A^*} A^* T_{A^*}
= A T_A A^* T_{A^*} = S_A S_{A^*} 
= S_A S_A^*.\qedhere
\end{align*}
 \end{proof}

We say that two normal operators $A$ and $B$ 
are {\emdf strongly commuting}, if the bounded
operators $T_A, S_A, T_B, S_B$ are pairwise
commuting.

\begin{thm}[Spectral Theorem: General Case]\label{spt.t.unb}
Let $A_1, \dots, A_d$ be pairwise strongly commuting normal operators
on a Hilbert space $H$, all self-adjoint if $\K = \R$. 
Then there is a unique Borel calculus $(\Phi,H)$ on
$\K^d$ such that $\Phi(\bfz_j) = A_j$
for all $j=1, \dots, d$.  
\end{thm}

\begin{proof}
Uniqueness is clear by Theorem \ref{uni.t.Cd}.
For existence,  we apply Theorem \ref{spt.t.bdd} to 
the tuple
$(T_{A_1}, \dots, T_{A_d}, 
S_{A_1}, \dots, S_{A_d})$
to obtain a unique 
Borel functional calculus
$\Psi$ on $\K^{2d}$ such that 
\[  \bft_j = T_{A_j},\quad \bfs_j = S_{A_j},
\]
where $\bft_1, \dots \bft_d, \bfs_1,\dots,  \bfs_d$ are just the
coordinate projections. 

Since each $T_{A_j}$ is self-adjoint, positive, contractive  and 
injective, its associated
Borel calculus is concentrated 
on $(0, 1]$. It follows that 
$\Psi$ is concentrated on 
$(0,1]^d \times \K^d$. 
Define
\[ f : (0,1]^d \times \K^d 
\to \K^d,
\quad f(t_1,\dots, t_d,s_1, \dots, s_d)
:= (s_1/t_1, \dots, s_d/t_d)
\]
and let $\Phi$ be the push-forward 
functional calculus of $\Psi$ along $f$.
Then
\[ \Phi(\bfz_j) =  \Psi(\bft_j^{-1}\bfs_j)
= \Psi(\bft_j)^{-1} \Psi(\bfs_j)
= T_{A_j}^{-1} S_{A_j} = A_j.
\]
The proof is complete. 
\end{proof}

By Theorem \ref{spt.t.unb}, 
each normal (self-adjoint) operator, bounded or unbounded, on a
complex (real) Hilbert space $H$ comes with a
unique Borel calculus $\Phi_A$ on 
$\C$ ($\R$) such that $\Phi_A(\bfz) = A$. 
We call this the Borel calculus
{\emdf for $A$} or {\emdf associated with
$A$}. As before, one writes
\[ f(A) := \Phi_A(f)\qquad (f\in \Meas(\K)).
\]
The Borel calculus for $A$ is concentrated
on $\spec(A)$ (Corollary \ref{spc.c.single}) and, as in the bounded case, isolated spectral points
must be eigenvalues. The composition
rule
\[ g(f(A)) = (g\nach f)(A)
\]
now holds universally, for the same reason
as in the bounded operator case.

\medskip

\subsection{Strong Commutativity}

Formally, our notion of strong commutativity 
differs from that of Schmüdgen  from   
\cite{SchmuedgenUOH}. Instead of the pair of operators
$(T_A, S_A)$, Schmüdgen  employs the notion of the   
{\emdf bounded transform}
\[ Z_A := \Bigl( \frac{\bfz}{\sqrt{1 + \abs{\bfz}^2}}\Bigr)(A)
\]
of a normal operator $A$. 
Alternatively, one can write
\[ Z_A = A \bigl( (1 + A^*A)^{-1} \bigr)^\frac{1}{2},
\]
where the square root is defined via the
(continuous) functional calculus
for the self-adjoint operator
$(1 + A^*A)^{-1}$.   The following
proposition is the major step to showing that both notions
of strong commutatitvity agree.

\begin{prop}\label{spt.p.strongcom}
Let $A$ be a normal operator on $H$, self-adjoint if $\K= \R$.  
Then for $B \in \BL(H)$ the following assertions
are equivalent:
\begin{aufzii}
\item $BA \subseteq AB$,  i.e., $B$ commutes with $A$. 
\item $Bf(A) \subseteq f(A)B$ for all 
$f\in \Meas(\K)$.
\item $B$ commutes with $Z_A$.
\item $B$ commutes with $T_A$ and $S_A$.
\end{aufzii}
If $B$ is also normal (self-adjoint if $\K =\R$), then {\rm (i)-(iv)}
are also equivalent to the following
assertions:
\begin{aufzii} \setcounter{aufzii}{4}
\item $A$ and $B$ are strongly commuting.
\item $Z_A$ and $Z_B$ commute. 
\end{aufzii}
\end{prop}

\begin{proof}
The set  $E:= E(\Phi_A, \Phi_A; B) = 
\{ f\in \Meas(\K) \suchthat Bf(A) \subseteq f(A)B\}$ 
has the properties 1)--8) of Theorem \ref{uni.t.inter}.
Hence, by Lemma \ref{uni.l.Cd}\, (i)--(iv) are pairwise
equivalent.

Suppose, in addition, that $B$ is normal (self-adjoint if $\K = \R$). 
Then we may  
apply the already established equivalence (i)$\gdw$(iii)
to $B$ in place of $A$ and $Z_A$ in place of $B$. This
yields the equivalence (iii)$\gdw$(vi) as it stands.

Similarly, the already established equivalence (i)$\gdw$(iv)
applied to $B$ in place of $A$ and $S_A, T_A$ in place of
$B$ yields the equivalence (iv)$\gdw$(v) as it stands. 
\end{proof}

As a corollary we obtain that two normal
operators strongly commute in our sense if and only if
they do in the sense of Schmüdgen from \cite{SchmuedgenUOH}.

\begin{cor}
Let $A, B$ be normal operators
on a Hilbert space $H$, and self-adjoint if $\K = \R$. 
Then the following
assertions are equivalent:
\begin{aufzii}
\item $A$ and $B$ are strongly commmuting.
\item $Z_A$ and $Z_B$ commute.
\item $f(A)$ commutes with $g(B)$ whenever
$f,g\in 
\Meas(\K)$ and one of the operators is bounded.
\item $f(A)$ commutes with $g(B)$ whenever
$f,g\in  \BMb(\K)$.

\item The projection-valued measures associated with
$A$ and $B$ commute.
\end{aufzii}
\end{cor}

\begin{proof}
We note the trivial or close-to-trivial implications
(iii)$\dann$(v)$\gdw$(iv)$\dann$(ii), (i). 

\prfnoi
(i)$\dann$(iii): 
Suppose that (i) holds and $f(A)$ is bounded. 
Then $T_B$ commutes with $T_A$ and  $S_A$. By Proposition 
\ref{spt.p.strongcom} applied with  $T_B$ instead of $B$,
$T_B$ commutes with each $f(A)$. The same
holds for $S_B$. Hence, if $f(A)$ is bounded,
we can apply Proposition \ref{spt.p.strongcom} again
(now with $B$ replaced by $f(A)$ and $A$ replaced by $B$)
and conclude that $f(A)$ commutes with $g(B)$ whatever $g$ is. This proves (iii).

\prfnoi
(ii)$\dann$(iii): This is similar as before. 
\end{proof}

\begin{rem}
The definition of the bounded transform goes
back to  \cite{Woronowicz1991}.   
Schmüdgen \cite[Chapter 5]{SchmuedgenUOH}  
uses the bounded transform for passing from 
bounded to unbounded normal operators in the proof
of Theorem \ref{spt.t.unb}, cf.{ }also 
\cite{BuddeLandsman2016}. The advantage is
that to cover the case of a  single unbounded operator one
only needs the result for a single bounded operator, 
and this
may be helpful in a teaching context. On the other hand,
a nontrivial concept of functional calculus,
the square root, is needed 
to define the bounded transform in the
first place, whereas one has direct access  
to the operators $T_A$ and $S_A$ used in our approach. 
\end{rem}

\appendix

%\begin{thebibliography}{EKRW01}
%\end{thebibliography}
%\bibliographystyle{acm}
%\bibliography{lfcbiblio}

\def\cprime{$'$} \def\cprime{$'$} \def\cprime{$'$} \def\cprime{$'$}

\medskip

\subsection*{Acknowledgements}

In preliminary form, parts of this work were included in 
 the lecture notes for  the
21st International Internet Seminar on ``Functional Calculus'' 
during the academic year 2017/2018. I am indebted to the participating
students  and  colleagues, in particular to Jan van Neerven (Delft), 
Hendrik Vogt (Bremen) and, above all, to Jürgen Voigt (Dresden) for valuable remarks and
discussions.

This work was completed while I was spending a research
sabbatical at UNSW in Sydney. I am grateful to
Fedor Sukochev for his kind invitation. Moreover, 
I gratefully acknowledge the financial support from the DFG,
 project number 431663331.
\end{document}